\documentclass{article} 
\usepackage{amsmath,amsfonts,amsbsy,amsthm,latexsym,amssymb,enumerate}

\newtheorem{theorem}{Theorem}[section] 
 
\newtheorem{lemma}[theorem]{Lemma}

\begin{document} 
\sloppy 
\title{~\\[-6ex] Conditional limit theorems for\\ intermediately subcritical branching processes\\ in random environment\thanks{This paper is a part of the  research project 'Branching processes and random walks in random environment'
 supported by the German Research Foundation (DFG) and the 
 Russian Foundation of Basic Research (RFBF, Grant DFG-RFBR 08-01-91954)}} 
\author{\textsc{V.I. Afanasyev}\thanks{Department of 
 Discrete Mathematics, Steklov Mathematical Institute, 8 Gubkin Street, 
 119\,991 Moscow, Russia, viafan@mail.ru, vatutin@mi.ras.ru} 
    \hspace{.8cm} 
 \textsc{Ch. B\"oinghoff}\thanks{Fachbereich Mathematik, Universit\"at Frankfurt, Fach 
     187, D-60054 Frankfurt am Main, Germany, boeinghoff@math.uni-frankfurt.de, kersting@math.uni-frankfurt.de} \\ 
   \textsc{G. Kersting}$^\ddagger$
      \hspace{.8cm} \textsc{V.A. Vatutin}$^\dagger$}
 \maketitle 
\begin{abstract} 
For a branching process in random environment it is assumed that the offspring distribution of the individuals varies in a random fashion, independently from one generation to the other. For the subcritical regime a kind of phase transition appears. In this paper we study the intermediately subcritical case, which constitutes the borderline within this phase transition. We study the asymptotic behavior of the survival probability.  Next the size of the population  and the shape of the random environment conditioned on non-extinction is examined. Finally we show that conditioned on non-extinction periods of small and large population sizes alternate. This kind of 'bottleneck' behavior   appears under the annealed approach only in the intermediately subcritical case.
\end{abstract}  
\begin{small} 
 \end{small} 
\thispagestyle{empty} 
\section{Introduction and main results} 
For a branching process in random environment (BPRE), as introduced in \cite{at,sm}, it is assumed that the offspring distribution of the individuals varies in a random fashion, independently from one generation to the other. Conditioned  on the environment individuals reproduce independently of  each other. Let $Q_{n}$ be the random offspring distribution of an individual at generation~$n-1$ and let $Z_n$ denote the number of individuals at generation $n$. Then $Z_{n}$ is the sum of $Z_{n-1}$ independent random variables, each of which has distribution~$Q_{n}$. To give a formal definition let $\Delta$ be the space of      probability measures on \mbox{$\mathbb{N}_{0}= \{0,1,\ldots\}$}, which equipped with the metric of     total variation is a Polish space. Let $Q$ be a random variable taking values in $\Delta$. Then an infinite sequence $\Pi=(Q_{1}, Q_{2},\ldots)$ of i.i.d.\ copies of  $Q$ is said to form a \emph{random environment}. A sequence of $\mathbb{N}_0$-valued random variables  $Z_{0},Z_{1},\ldots$ is called a 
 \emph{branching process in the random environment} $\Pi$, if $Z_{0}$ is independent of $ \Pi$   and   given $ \Pi$ the  process $Z=(Z_{0},Z_{1},\ldots)$ is a Markov chain with 
\begin{equation}  \label{transition} 
    \mathcal{L} \big(Z_{n} \; \big| \; Z_{n-1}=z, \, \Pi 
   = 
     (q_{1},q_{2},\ldots) \big) \ = \ q_{n}^{*z} 
\end{equation} 
for every $n\in \mathbb{N}=\{1,2,\ldots\}$, $z \in \mathbb{N}_0$ and $q_{1},q_{2},\ldots \in \Delta$, where 
$q^{*z}$  is the $z$-fold convolution of the measure $q$. The   corresponding probability  measure on the underlying probability space will be denoted by $\mathbb{P}$. In the following 
we assume  that the process starts with a single founding ancestor, $Z_0=1$ a.s.,  and (without loss of generality) that $\mathbb{P}\big(Q( 0 )=1\big)  =0 $. (We shorten  $Q(\{y\}), q(\{y\})$ to $Q(y),q(y)$.)  Note that in general~$Z$ is not the superposition of $Z_0$ independent copies of the  process started at $Z_0=1$. 
 
It turns out  that the asymptotic behavior of the generation size  process $Z$ is  determined  in the main 
 by the associated random walk $S=(S_n)_{n\ge 0}$. This random walk has initial state 
$S_{0}=0$ and  increments   $X_n=S_n-S_{n-1}, \, n\ge 1$ defined as 
\[ 
    X_{n} \ = \ \log m(Q_n),  
\] 
where 
\[ 
m(q)\ = \
\sum_{y=0}^{\infty} y  q(y) 
\] 
is the mean of the offspring distribution $q \in \Delta$. 
In view of (\ref{transition}) and the assumption $Z_0=1$ a.s., 
      the conditional expectation of 
 $Z_{n}$      given the environment~$ \Pi$ 
   can be expressed by means of $S$ as 
\begin{equation*} 
  \mathbb{E}[Z_{n} \,| \,  \Pi \,]   \ = \ 
\prod_{k=1}^n m(Q_k)\ = \ \exp(S_n) \quad 
\mathbb{P}\text{--}a.s. 
\end{equation*} 
Averaging over the environment gives 
\begin{align}   
  \mathbb{E}[Z_{n} ]   \ = \ 
 \mathbb{E}[ m(Q) ]^n.  \label{expect2} 
\end{align} 

  If the random walk $S$ drifts to $-\infty$,  then the branching process 
 is said to be   \emph{subcritical}. In case $X=\log  m(Q)$ has finite 
mean, subcriticality corresponds to $\mathbb{E}[X]<0.$ 
 For such processes 
  the conditional non-extinction probability at $n$ 
  decays at an exponential rate for almost every environment. 
This fact is an immediate consequence of the strong law of large numbers and the first moment  
estimate 
\begin{eqnarray} 
   \lefteqn{ \mathbb{P} (Z_{n} > 0 \,| \,  \Pi ) \ =  \ 
\min_{0\le k \leq n} 
       \mathbb{P} (Z_{k} > 0 \, | \,  \Pi )}     \nonumber\\ 
    &\leq & 
  \min_{0\le k \leq n}   \mathbb{E}[ Z_{k} \,| \,  \Pi ]  \ = \ 
\exp\big(\min_{0\le k \leq n} S_{k}\big)  \quad 
\mathbb{P}\text{--}a.s.       \label{gleichung1} 
\end{eqnarray} 

As was  observed by Afanasyev~\cite{af_80} 
   and later independently by Dekking~\cite{de} there are three possibilities for the asymptotic behavior of subcritical branching processes. They are called the \emph{weakly} subcritical, the \emph{intermediately}  subcritical and the 
\emph{strongly} subcritical case.  
The present article is a part of several publications
having started  with \cite{abkv,agkv,agkv2}, in which we try to develop characteristic properties of the different  cases. For a comparative discussion 
 we refer the reader to \cite{bgk}.
 
Here we study the intermediate  case. It is located at the borderline between the weakly and strongly subcritical cases. 
The passage corresponds to a phase transition in the model, thus a particular rich behavior can be expected for the intermediate case. 
This is reflected in our results below. In particular we shall observe a kind of bottleneck phenomenon, which does not occur elsewhere under the annealed approach. Similar behavior
has been noticed under the  quenched approach in the critical regime (see \cite{dy04}, \cite{dy05} and \cite{dy07}).

 \paragraph{Assumption A1. } {\em The process $Z$ is intermediately subcritical, i.e.}
  \[  \mathbb{E} [ Xe^{ X}] \ = \ 0 \ . \] 

\noindent
 The assumption suggests to change from $\mathbb P$  to a measure $\mathbf{P}$: For every $n\in\mathbb{N}$ and every bounded, measurable function $\varphi:\Delta^n\times \mathbb{N}_0^{n+1}\rightarrow\mathbb{R}$,  $\mathbf{P}$ is given by its expectation
\[ \mathbf{E}[ \varphi(Q_1,\ldots,Q_n, Z_0,\ldots,Z_n)] \ = \ \gamma^{-n} \mathbb{E}\big[\varphi(Q_1,\ldots,Q_n, Z_0,\ldots,Z_n)e^{(S_n-S_0)}\big] \ , \] 
with
\[ \gamma \ = \ \mathbb{E} [e^{X}] \ . \]  
(We include $S_0$ in the above expression, because later on we shall also consider cases where $S_0 \neq 0$.) From \eqref{expect2} we obtain 
\begin{align} \mathbb E[Z_n]= \gamma^n \ . 
\label{expect}
\end{align}
The assumption $ \mathbb{E}[Xe^{X}]=0$ translates into
\[ \mathbf{E}[X] \ = \ 0 \ . \]
Thus $S$ becomes a recurrent random walk under $\mathbf{P}$. 

As to the regularity of the distribution of $X$ we make the following assumptions.

\paragraph{Assumption A2.} {\em The distribution of $X$ has  finite variance with respect to $\mathbf{P}$ or  (more generally) belongs to the domain of attraction of some stable law with index $\alpha \in (1,2]$. It is non-lattice.} \\

\noindent
Since $\mathbf{E}[X]=0$ this means that there is an increasing sequence of positive numbers
\[ a_n  \ = \ n^{1/\alpha} \ell_n \]
with a slowly varying sequence $\ell_1,\ell_2,\ldots$ such that for $n\rightarrow\infty$
\[ \mathbf{P}\big(\tfrac 1{a_n}S_n \in dx \, \big) \ \to \ s(x) \, dx\] 
weakly, where $s(x)$ denotes the density of the limiting stable law. Note that due to the change of measure $X^-$ always has finite variance and an infinite variance may only arise from $X^+$. In case of $\alpha < 2$ this is the so-called spectrally positive case (\cite{bi}, Section 8.2.9).

Our last assumption on the environment concerns the standardized truncated second moment of $Q$,
\[ \zeta(a) \ = \ \frac 1 {m(Q)^2}\sum_{y=a}^\infty y^2 Q(y)  \ , \quad a \in \mathbb{N} \ . \]

\paragraph{Assumption A3.} {\em For some $\varepsilon > 0$ and some $a \in \mathbb{N}$}
\[ \mathbf{E} [ (\log^+ \zeta(a))^{\alpha + \varepsilon}] \ < \ \infty \ ,\]
where $\log^+ x= \log(x \vee 1)$.\\\\
See \cite{agkv} for examples where this assumption is fulfilled. In particular our results hold for binary branching processes in random environment (where individuals have either two children or none) and for cases where $Q$ is a.s. a Poisson distribution or a.s. a geometric distribution.\\\\
The following theorem has been obtained under quite stronger assumptions in \cite{af_80,gkv,va_04}. 
Let
\[ \tau_n = \min\{ k \le n \mid S_k \le S_0, S_1, \ldots,S_n\} \]
be the moment, when $S_k$ takes its minimum within $S_0$ to $S_n$ for the first time.
\begin{theorem}\label{survival}
Under Assumptions A1 to A3, there is a constant $0<\theta<\infty$ such that as $n\rightarrow\infty$
\begin{align}
 \mathbb{P}(Z_n>0) \sim \theta \gamma^n \mathbf{P}(\tau_n=n). \nonumber
\end{align}
\end{theorem}
\noindent 
In this form the result holds in the strongly subcritical case too \cite{gl}, however it differs from the corresponding result in the weakly subcritical case \cite{abkv}.
Along the way of proving the subsequent results we also obtain a proof of the above theorem. 
Since $\mathbf{P}(\tau_n=n) \sim 1/b_n$ with
\[ b_n = n^{1-\alpha^{-1}} \ell_n'\]
for some slowly varying sequence $(\ell_n')$ (see Lemma \ref{maxim} below), it follows 
\[  \mathbb{P}(Z_n>0) \sim \theta\frac{\gamma^n}{b_n}  \ .\]

The next theorem deals with the branching process conditioned on survival at time $n$.
\begin{theorem} \label{theoZ_n}
Under Assumptions A1 to A3 the distribution of $Z_n$ conditioned on the event $Z_n >0$ converges weakly to a probability distribution on $\mathbb N$. Also for every $\beta < 1$ the sequence $\mathbb E[Z_n^\beta \mid Z_n>0]$ is bounded .
\end{theorem}
\noindent
For $\beta=1$ this statement is no longer true, since $\mathbb E[Z_n]= \gamma^n$ from \eqref{expect} and consequently $\mathbb E[Z_n \mid Z_n>0] \to \infty$ for $n \to \infty$.

The next theorem captures the typical appearence of the random environment, when conditioned on survival.  Let $S^n$ be the stochastic process with paths in the Skorohod space $D[0,1]$ of c\`adl\`ag functions on $[0,1]$ given by
\[ S^n_t  =  S_{nt} \ , \quad 0 \le t \le 1 \ . \]
We agree on the convention $S_{nt}=S_{ \lfloor nt\rfloor }$, which we use correspondingly for $Z_{nt}, \tau_{nt}$. Also let $L^*$ denote a L\'evy-process on $[0,1]$ conditioned to attain its minimum at time $t=1$. The precise definition will be given in Section 2.
\begin{theorem}\label{limitlaw} 
Assume Assumptions A1 to A3. Then, as $n\rightarrow\infty$,  the distribution of  $n-\tau_n$ conditioned on the event $ Z_n>0$ converges to a probability distribution $p$ on $\mathbb N_0$  and
\begin{eqnarray}
  \big(\tfrac 1{a_n} S^n\ \big|\ Z_n>0\big) &\stackrel{d}{\rightarrow} &   L^* \ \nonumber
\end{eqnarray}
in the Skorohod space $D[0,1]$. Also both quantities are asymptotically independent, namely for every bounded continuous $\varphi:D[0,1] \to \mathbb R$ and every $B \subset \mathbb N_0$
\[ \mathbb E\big[ \varphi(\tfrac 1{a_n}S^n); n-\tau_n \in B   \mid Z_n > 0\big]\to \mathbf E[\varphi(L^*) ]p(B) \ . \]
\end{theorem}
The first statement also holds for strongly subcritical, but not for critical or weakly subcritical BPRE. The   limit $L^*$ only appears in the intermediate case. 

The last theorem characterizes the typical behavior of $Z$, conditioned on survival. For a partial result see Theorem 1 in \cite{af_01}. Recall that $\tau_{nt}$ is the moment, when $S_0, \ldots, S_{nt}$ takes its minimum.

\begin{theorem}\label{theomain} Let $0 < t_1 < \cdots < t_r < 1$. For $i=1, \ldots,r$ let
\[ \mu(i) = \min\big\{ j \le i: \inf_{t \le t_j}L_t^* = \inf_{t \le t_i}L_t^*\big\} \ .\]
Then under Assumptions A1 to A3 there are i.i.d. random variables $V_1, \ldots,V_r$ with values in $\mathbb N$  and independent of $L^*$ such that
\begin{align*}
  \big((Z_{ \tau_{nt_1}},\ldots,Z_{\tau_{nt_r}} ) \mid Z_n>0\big) 
 \stackrel{d}{\rightarrow}  (V_{ \mu(1)},\ldots,V_{ \mu(r)})  
\end{align*}
as $n \to \infty$. Also there are i.i.d. strictly positive random variables $W_1,  \ldots,W_r$ independent of $L^*$ such that
\begin{align*}
  \Big(\big( \frac{Z_{  n t_1 }}{e^{S_{  n t_1 }-S_{\tau_{nt_1}}} }  ,\ldots, \frac{Z_{  nt_r }}{e^{S_{  n t_r }-S_{\tau_{nt_r}}}}  \big) \ \big|\ Z_n>0\Big)   \stackrel{d}{\to}  (W_{ \mu(1)},\ldots,W_{ \mu(r)})   
\end{align*}
as $n \to \infty$.
\end{theorem}

For $r=1$ and $t_1=t$ the theorem says the following: At time $\tau_{nt}$ the population  consists only of 
few individuals, whereas at time $nt$ it is large, namely of order $e^{S_{  n t }-S_{\tau_{ n t }}}$-many individuals, which for every $\varepsilon>0$ is bigger than $e^{\delta a_n}$ with probability $1- \varepsilon$, if $\delta>0$ is small enough. Thus      the minimum of the random walk at time $\tau_{nt}$ acts as a bottleneck for the population, whereas afterwards the increasing random walk generates an environment, which is favorable for growth.

Moreover: In case of $r=2$ either $\tau_{nt_1}<\tau_{nt_2}$ or $\tau_{nt_1}=\tau_{nt_2}$, which for the limiting process $L^*$ means $\mu(2)=2$ or $\mu(2)=1$. The theorem says that in the first situation of two bottlenecks the population sizes $Z_{ \tau_{nt_1}}$ and $Z_{\tau_{nt_2}}$ are asymptotically independent, as well as the sizes $Z_{  n t_1 }$ and $Z_{  n t_2 }$. In the second situation of one common bottleneck certainly $Z_{ \tau_{nt_1}}$ and $Z_{\tau_{nt_2}}$ are equal. Interestingly this is asymptotically true as well for $Z_{  n t_1 }/e^{S_{  n t_1 }-S_{\tau_{nt_1}}} $ and $Z_{  n t_2 }/e^{S_{  n t_2 }-S_{\tau_{nt_2}}} $. Here  a law of large numbers is at work, in a similar fashion as for supercritical Galton-Watson processes. 

As a corollary of Theorem \ref{limitlaw} and \ref{theomain} we observe that $\big(\tfrac 1{a_n} \log Z_{nt}\big)_{0 \le t \le 1}$ converges to a L\'evy-process, conditioned to take its minimum at the end and reflected at zero. For the finite dimensional distributions this follows from the theorems together with path properties of L\'evy-processes.

The proofs rest largely on the fact that the event $Z_n>0$ asymptotically entails that $\tau_n$ takes a value close to $n$, as stated in Theorem \ref{limitlaw}. Thus it is our strategy to replace the conditioning event $Z_n>0$ by events $\tau_n=n-m$, which are easier to handle. Here we can build on some random walk theory. For the proof of the last theorem we also make use of constructions of {\em trees with stem} going back to Lyons, Perez and Pemantle \cite{lpp} and Geiger \cite{ge_99} for Galton-Watson processes. They establish a connection between branching processes conditioned to survive and branching processes with immigration.

The paper is organized as follows: In Section 2 we compile and prove several results on random walks. In Section 3 the proofs of the first three theorems are given. Section 4 deals with trees with stem and Section 5 contains the proof of our last theorem.

\section{Results on random walks}

In this section we assemble several  auxiliary results on the random walk $S$. We allow for an arbitrary initial value $S_0=x $. Then we write  $\mathbf P_x(\cdot)$ and $\mathbf E_x[\cdot]$ for the corresponding probabilities and expectations. Thus $\mathbf P= \mathbf P_0$.

\paragraph{2.1 Some asymptotic results} 
Let us introduce for $n \ge 1$
\[   L_n \ = \ \min(S_1, \ldots, S_n) \ , \quad M_n \ = \ \max(S_1, \ldots,S_n)\]
and as above for $n \ge 0$
\[ \tau_n \ = \  \min \{ k \le n : S_k =\min(0, L_n) \} \ . \]
There is a   connection between $M_n$ and $\tau_n$, set up by the dual random walk 
\[ \hat S_k \ =\ S_n-S_{n-k} \ , \quad 0 \le k \le n\ . \]
Namely $\{\tau_n=n\}=\{ \hat M_n< 0 \} $ with $\hat M_n=\max(\hat S_1, \ldots,\hat S_n)$ and consequently
\[ \mathbf P(\tau_n=n)=\mathbf P(M_n< 0) \ . \] 
In particular $\mathbf P(\tau_n=n)$ is decereasing.

Next define the renewal functions $u: \mathbb{R} \to \mathbb{R}$ and $v: \mathbb{R} \to \mathbb{R}$ by
\begin{align*} u(x)  \ &= \ 1 + \sum_{k=1}^\infty \mathbf{P} (-S_k\le x, M_k < 0 ) \ , \quad x \ge 0 \ ,\\
v(x) \ &= \ 1 + \sum_{k=1}^\infty \mathbf{P} (-S_k > x, L_k > 0) \ , \quad x < 0 \  ,\\
v(0) \ &= \  \mathbf E[ v(X); X<0] 
\end{align*}
and 0 elsewhere. In particular $u(0) =1$. It is well-known that $0<v(0) \le 1$, for details we refer to \cite{ca}, Appendix B and \cite{dy05}. (Our function $v(x)$ coincides with the function $v(x)$ in \cite{abkv} up to a constant.) Also $u(x) $ and $v(-x)$ are of order $O(x)$ for $x \to \infty$.

\begin{lemma} \label{pro1}
Under Assumption A2 there is for every $r>0$ a $\kappa >0$ such that
\[ \mathbf E [e^{-rS_n}; L_n \ge 0 ]  \ \sim \ \kappa n^{-1}a_n^{-1}  \]
as $n \to  \infty$.
\end{lemma}
\noindent
For the proof   we refer to Proposition 2.1 in \cite{abkv}.
\begin{lemma} \label{maxim}
Under Assumption A2 there are real numbers
\[b_n= n^{1-\alpha^{-1}}\ell_n' \ , \quad  n \ge 1\] with a sequence $(\ell_n')$ slowly varying at infinity such that for every $x \ge 0$
\[ \mathbf P(M_n < x) \ \sim \ v(-x)b_n^{-1}  \]
as $n \to \infty$. Also there is a constant $c>0$ such that for all $x \ge 0$
\[ \mathbf P(M_n < x)=\mathbf P_{-x}(M_n < 0) \le cv(-x)b_n^{-1} \ . \]
\end{lemma}
\begin{proof}
The corresponding statements for $\mathbf P(M_n \le x)$ are well-known. Indeed the first one is contained in Theorem 8.9.12 in \cite{bi}, where $\rho$ 
now is equal to $1- \alpha^{-1}$, since we are in the spectrally positive case  (note that the proof therein works for all $x \ge 0$ and not only, as stated, for the continuity points of $v$). 

For $x>0$ this proof completely translates to $\mathbf P(M_n < x)$. Therefrom the case $x=0$ can be treated as follows:
\begin{align*}
\mathbf P(M_n < 0) &= \mathbf E\big[ \mathbf P_{X_1}(M_{n-1} < 0); X_1 < 0\big] \\ &= b_{n-1} \mathbf E\Big[ \frac{\mathbf P_{X_1}(M_{n-1} < 0)}{b_{n-1}}; X_1 < 0\Big] \ . 
\end{align*}
From dominated convergence and from $b_n \sim b_{n-1}$ we get
\[ \mathbf P(M_n < 0)   \sim  b_n \mathbf E[v(X_1);  X_1 < 0] \ . \]
Now from equation \eqref{harm} below the right-hand side is equal to $v(0)$, as defined above, which gives the claim.

The second statement is obtained just as in Lemma 2.1 in \cite{agkv}. 
\end{proof}

\paragraph{2.2 The probability measures $\mathbf{P}^+$ and $\mathbf{P}^-$} The fundamental properties of $u,v$ are the identities
\begin{equation} \begin{array}{rl} \mathbf{E}  [u(x+X); X + x \ge 0] \ = \ u(x) \ ,  &x \ge 0 \ ,  \\   \mathbf{E} [ v(x+X);X+x<0] \ = \ v(x) \ ,   &x \le 0 \ , \end{array}  \label{harm}
\end{equation}
which hold for every oscillating random walk (see e.g. \cite{dy05}). 
It follows that $u$ and $v$ give rise to further probability measures $\mathbf{P}^+$ and $\mathbf{P}^-$. The construction procedure is standard and explained for $\mathbf{P}^+$ in detail in \cite{bedo,agkv}. We shortly summarize the procedure.

Consider the filtration $\mathcal F=(\mathcal F_n)_{n \ge 0}$, where $\mathcal F_n=\sigma(Q_1,\ldots,Q_n,Z_0,\ldots,Z_n)$. Thus $S$ is adapted to $\mathcal F$ and $X_{n+1}$ (as well as $Q_{n+1}$) is independent ot $\mathcal F_n$ for all $n \ge 0$. Then for every bounded, $\mathcal F_n$-measurable random variable $R_n$
\begin{align*} \mathbf E^+_x[ R_n ] \ &= \ \frac{1}{u(x)}\mathbf{E}_x[  R_n u(S_n); L_n \ge 0]\ , \quad x \ge 0 \ ,\\ \mathbf E^-_x[ R_n ] \ &= \ \frac{1}{v(x)}\mathbf{E}_x[ R_n v(S_n); M_n < 0] \ , \quad x \le 0 \ . 
\end{align*}
These are Doob's transforms from the theory of Markov chains. Shortly speaking $\mathbf P_x^+$ and $\mathbf P_x^-$ correspond to conditioning the random walk $S$ not to enter $(-\infty,0)$ and $[0,\infty)$ respectively. 

The following lemma is taken from \cite{bedo,agkv}.
\begin{lemma} \label{limitEminus}
Assume A2 and let $U_1,U_2, \ldots$ be a sequence of uniformly bounded random variables, adapted to the filtration $\mathcal F$.
If $U_n \to U_\infty$ $\mathbf P^+$-a.s. for some limiting random variable $U_\infty$, then as $n \to \infty$
\[ \mathbf E[ U_n \mid L_n \ge 0] \ \to \ \mathbf E^+[U_\infty] \ . \]
Similarly, if  $U_n \to U_\infty$ $\mathbf P^-$-a.s., then as $n \to \infty$
\[ \mathbf E[ U_n \mid M_n < 0] \ \to \ \mathbf E^-[U_\infty] \ . \]
\end{lemma}
\noindent
The first part coincides with Lemma 2.5 from \cite{agkv}. The proof of the second part follows exactly the same lines using Lemma \ref{maxim}.

\paragraph{2.3 Two functional limit results} 

Because of Assumption A2 there exists a L\'evy-process $L=(L_t)_{t \ge 0}$ such that the processes $S^n=(\tfrac 1{a_n} S_{  nt  })_{0\le t \le 1}$ converge in distribution to $L$ in the Skorohod space $D[0,1]$. Let $L^-=(L^-_t)_{0 \le t \le 1}$ 
denote the corresponding non-positive L\'evy-meander. This is the process $(L_t)_{0 \le t \le 1}$, conditioned on the event $\sup_{ t \le 1} L_t  \le 0$ (see  \cite{chau} and \cite{be}).

\begin{lemma} \label{funclimit2}
Under Assumptions A1 and A2 for every $x\geq 0$
\begin{equation*}
\big( \tfrac{1}{a_{n}}S^{n}\mid M_{n}<-x\big) \overset{d}{\rightarrow }L^{-}
\end{equation*}
in the Skorohod space $D[0,1]$.
\end{lemma}
\noindent
The proof follows exactly the same arguments  as the proof of Lemma 2.3 in \cite{agkv}, i.e. using  the suitably adapted decomposition (2.10) therein and \cite{do}.

From $L^-$ we obtain the process $L^*$ as follows.  
Let $\Lambda: D[0,1] \to D[0,1]$ be the mapping 
$ g \mapsto \hat g $ 
given by
\[   \hat g(t)= g(1)- g(s-) \ , \  0 \le t \le 1   \ , \quad \text{with } s=1-t  \]
and $g(0-)=0$.  $\Lambda$ is a continuous mapping and $\Lambda^{-1}=\Lambda$. Note that $\Lambda$ maps the subset  $D^-=\{ g\in D[0,1]: \sup_{ t \ge \varepsilon} g(t) < 0  \text{ for all } \varepsilon > 0\}$ onto the set $D^*=\{g\in D[0,1] : \inf_{ s \le 1-\varepsilon} g(s)> g(1)  \text{ for all } \varepsilon > 0 \}$.

Now let
\[ L^* = \Lambda(L^-) \ . \]
Since $L^- \in D^-$ a.s. it follows that $L^* \in D^*$ a.s. This means that $L^*$ takes its infimum at the end a.s. $L^*$ may be viewed as the process $(L_t)_{0 \le t \le 1}$, conditioned to attain its infimum at $t=1$. This becomes clear from the following result.

\begin{lemma} \label{funclimit}
Under Assumptions A1 and A2  
\[  \big(\tfrac 1{a_n}S^{n} \mid \tau_n=n\big) \ \stackrel{d}{\rightarrow}  \ L^*\]
in $D[0,1]$.
\end{lemma}

\begin{proof}
We may replace $S^n$ by the process $T^n$ given by  $T^n_t= S_{t+1/n}^n$ for $t\le1- \tfrac 1n$ and $T^n_t=S^n_1$ for $1-\tfrac 1n < t \le 1$. From Lemma \ref{funclimit2} 
\[ \big(\tfrac 1{a_n}\Lambda(T^{ n}) \mid M_{n} < 0\big) \ \stackrel{d}{\rightarrow}  \ L^*  \ . \]
Now $\Lambda(T^n)$ is obtained from $S^n$, if we just interchange the jumps in $S^n$ from $X_1, \ldots, X_n$ to $X_n \ldots, X_1$. This corresponds to proceeding to the dual random walk, and it follows
\[ \big(\tfrac 1{a_n}S^n \mid \tau_n=n\big) \ \stackrel{d}{\rightarrow}  \ L^*  \ . \]
This is the claim.
\end{proof}

We end this section by some remarks on the distribution of $L^*_1$. First $L_1$ has a stable distribution, thus it has a density with respect to Lebesgue measure and is unbounded from below. Since we are in the spectrally positive case, $L$ has no negative jumps a.s. Therefore we may use fluctuation theory for the process $L^\downarrow$, which is the 
L\'evy-process, conditioned to take values in $(-\infty,0]$, see \cite{be}, Section VII.3. From Corollary 16 therein it follows that $L^\downarrow_1$ has a density  and is unbounded from below, too. From \cite{chau} (see also \cite{ca}) it follows that the distributions of  $L^\downarrow_1$ and $L^*_1$ are mutually absolutely continuous, therefore also the distribution $\nu$ of $L^*_1$ has a density and is not concentrated on some compact interval.

\paragraph{2.4 Further limit results} 
Let $Q_j=Q_1$ for $j \le 0$.
 
\begin{lemma} \label{leQ}
Under Assumptions A1 and A2 for $m \ge 0,k \ge 1$ the distribution of  
\begin{align*} \Big( \big(Q_{\tau_n+1},\ldots, Q_{\tau_n+k}\big), \big(Q_{\tau_n},\ldots, Q_{\tau_n-k+1}\big),\frac{(S_{\tau_n}, S_{n-m})}{a_n} \Big)  
\end{align*}
converges weakly to a probability measure $ \mu_k' \otimes \mu_k'' \otimes\mu$, where  $\mu_k'$, $\mu_k''$ are the distributions 
of $(Q_1, \ldots,Q_k)$ under the probability measures $\mathbf P^+$, $\mathbf P^-$ and $\mu$ is a nondegenerate probability measure on $\mathbb R^2$.
\end{lemma}
\begin{proof} 
Let for $ r \ge 0$  
\begin{equation*}
Q^+(r) =( Q_{r+1},...,Q_{r+k}) \ , \quad 
Q^-(r)=( Q_{r},...,Q_{r-k+1}) \ .
\end{equation*}
Let $\phi_1,\phi_2:\Delta^k\to \mathbb R $ be bounded functions and $\phi_3,\phi_4: \mathbb R\to \mathbb R$ be bounded continuous functions. A decomposition with respect to $\tau _{n}$ yields
\begin{align}
&\mathbf{E}\big[ \phi _{1}( Q^-( \tau _{n} )) \phi _{2}( Q^+( \tau _{n} ) 
) \phi _{3}( \tfrac{S_{\tau _{n}}}{a_{n}})  \phi _{4}( \tfrac{%
S_{n-m}-S_{\tau _{n}}}{a_{n}}) \big]  \notag \\
&=\sum_{r=0}^{n}\mathbf{E}\big[ \phi _{1}( Q^-(
r ) ) \phi_{2}( Q^+( r ) )\phi _{3}( \tfrac{S_{r}}{a_{n}})  \phi _{4}( \tfrac{%
S_{n-m}-S_{r}}{a_{n}}) ;\tau _{n}=r\big]\ .  \label{Dec1}
\end{align}%
Letting
$ L_{r,n} = \min (S_{r+1}, \ldots, S_n)-S_r$
and using duality we get for $r >k$
\begin{align*}
 &\mathbf{E}[ \phi _{1}( Q^-(r) ) \phi _{2}( Q^+(r) )
\phi _{3}( \tfrac{S_r}{a_{n}}) \phi _{4}( \tfrac{%
S_{n-m}-S_r}{a_{n}}) ;\tau _{n}=r] \\
&= \mathbf{E}[  \phi _{1}( Q^-(r)
)  \phi _{3}( \tfrac{S_r}{a_{n}})  \phi _{2}( 
Q^+(r) )  \phi _{4}( \tfrac{%
S_{n-m}-S_r}{a_{n}}) ;\tau_r=r,L_{r,n} \geq 0] \\
&= \mathbf{E}[  \phi _{1}( Q^-(r)
)  \phi _{3}( \tfrac{S_r}{a_{n}})  ;\tau_r=r] \mathbf{E}[ \phi _{2}( 
Q^+(0) )  \phi _{4}( 
\tfrac{S_{n-m-r} }{a_{n}}) ;L_{n-r} \geq 0]  \\
&= \mathbf{E}[ \phi _{1}( 
Q^+(0) )\phi _{3}( \tfrac{S_{r} }{a_{n}}%
) ;M_{r} <0] \mathbf{E}[ \phi _{2}( 
Q^+(0) ) \phi _{4}( 
\tfrac{S_{n-m-r} }{a_{n}}) ;L_{n-r} \geq 0] \ .
\end{align*}%
Moreover for $r>k$
\begin{align*}
&\frac{\mathbf{E}\big[ \phi _{1}( 
Q^+(0) )\phi _{3}( \tfrac{S_{r} }{a_{n}}%
) ;M_{r} <0\big] }{\mathbf P(M_r < 0)} \\
&\quad =  \mathbf{E}\big[ \phi _{1}( 
Q^+(0) )\mathbf E_{S_k} [ \phi_3(\tfrac{S_{r-k}}{a_n}) \mid M_{r-k}<0 ]\tfrac{\mathbf P_{S_k}(M_{r-k}<0)}{\mathbf P(M_r<0)} ; M_k < 0\big]\ .
\end{align*}
Therefore by Lemmas \ref{funclimit2}, \ref{maxim} and dominated convergence, if $r_n \sim tn$ for some $0<t<1$, then $a_{r_n}/ a_n  \sim t^{\frac 1\alpha}$ and
\begin{align*}
&\frac{\mathbf E\big[ \phi _1( 
Q^+(0) )  \phi _3( \tfrac{S_{r_n} }{a_n}%
) ;M_{r_n}<0\big] }{\mathbf P(M_{r_n} < 0)}\\ 
&\qquad\qquad\qquad\to \mathbf E \big[\phi _{1}( 
Q^+(0) )v(S_k); M_k<0\big] \, \mathbf E[\phi_3(t^{\frac 1\alpha}L_1^-)]\\
&\qquad\qquad\qquad = \mathbf E^- \big[\phi _{1}( 
Q^+(0) ) \big] \, \mathbf E[\phi_3(t^{\frac 1\alpha}L_1^-)] \ .
\end{align*} 
In much the same way, letting $L^+$ be the positive L\'evy meander and using Lemma 2.3 from \cite{agkv},   it follows that 
\begin{align*}
&\frac{\mathbf{E}[ \phi _{2}( 
Q^+(0) ) \phi _{4}( 
\tfrac{S_{n-m-r_n} }{a_{n}}) ;L_{n-r_n} \geq 0]  }{\mathbf P(L_{n-r_n} \geq 0)}\\ 
&\qquad\qquad\qquad \to \mathbf E^+ \big[\phi _{2}( 
Q^+(0) ) \big] \, \mathbf E[\phi_4((1-t)^{\frac 1\alpha}L_{1 }^+)] \ .
\end{align*}
Since $\mathbf P(M_{r_n} < 0)\mathbf P(L_{n-r_n} \geq 0)= \mathbf P(\tau_n=r_n)$, we obtain for $r_n \sim tn$ and $0<t<1$
\begin{align*}
 &\mathbf{E}[ \phi _{1}( Q^-(r_n)) \phi _{2}( Q^+(r_n) ) \phi _{3}( \tfrac{S_{r_n}}{a_{n}})  \phi _{4}( \tfrac{S_{n-m}-S_{r_n}}{a_{n}}) \mid \tau _{n}=r_n] \\
 &\ \to  \mathbf E^- \big[\phi _{1}( Q_{1} ,...,Q_{k} ) \big] \mathbf E^+ \big[\phi _{2}( Q_{1} ,...,Q_{k} ) \big] \mathbf E[\phi_3(t^{\frac 1\alpha}L_1^-)]\mathbf E[\phi_4((1-t)^{\frac 1\alpha}L_{1 }^+)] \ .
\end{align*}

Now in view of Assumption A2, the generalized arcsine law (see \cite{bi}) is valid for $\tau
_{n} $, i.e. $\tau _{n}/n$ is convergent in distribution to a Beta-distribution with a density, which we denote by $g(t)\, dt$. Therefore it follows from \eqref{Dec1} that
\begin{align*}
 \mathbf{E}\big[ &\phi _{1}( Q^-( \tau _{n} )) \phi _{2}( Q^+( \tau _{n} ) 
) \phi _{3}( \tfrac{S_{\tau _{n}}}{a_{n}})  \phi _{4}( \tfrac{%
S_{n-m}-S_{\tau _{n}}}{a_{n}}) \big] \\
&\quad \to  \mathbf E^- \big[\phi _{1}( Q_{1} ,...,Q_{k} ) \big] \mathbf E^+ \big[\phi _{2}( Q_{1} ,...,Q_{k} ) \big] \\&\qquad\qquad \mbox{}\times \int_0^1 \mathbf E[\phi_3(t^{\frac 1\alpha}L_1^-))]\mathbf E[\phi_4((1-t)^{\frac 1\alpha}L_{1 }^+)] \, g(t)\, dt \ .
\end{align*}
This gives the claim.
\end{proof}

\noindent
Next let $0=t_0<t_1 < \cdots < t_r < t_{r+1}=1$ and for $1 \le i \le r$
\begin{align}\sigma_{i,n} = \min \{k  : nt_{i-1}  \le k \le nt_i,  S_k \le S_j \text{ for all } nt_{i-1}  \le j \le nt_i  \} \label{sigma}
\end{align}
be the first moment, when $S_k$ takes its minimum between $nt_{i-1}$ and $nt_i$. 

\pagebreak

\begin{lemma} \label{cor}
Let $ m\ge 0$ and $k,r \ge 1$. Then under Assumptions A1,  A2, given the event $\tau_{n-m}=n-m$, the random elements in $\Delta^{2k}$
\[Q^{(i)}=\big( (Q_{\sigma_{i,n} +1},\ldots,Q_{\sigma_{i,n}+k}), (Q_{\sigma_{i,n}}, \ldots, Q_{\sigma_{i,n}-k+1})\big) \ , \quad i=1, \ldots, r \ , \]
are asymptotically independent with asymptotic distribution $\mu_k' \otimes \mu_k''$. Also, given $\tau_{n-m}=n-m$, they are asymptotically independent from the random vector
\[ \tfrac 1{a_n} (S_{\sigma_{1,n}}, S_{nt_1},\ldots,S_{\sigma_{r,n}}, S_{nt_r}) \ . \]
\end{lemma}

\begin{proof} Recall from above that, given $\tau_{n }=n $, the distribution of $\tfrac 1{a_n} S_{n } $ is weakly convergent to a probability measure $\nu$ on $\mathbb R^-$, the distribution of $L^*_1$, which possesses a density and is not concentrated on a compact interval. 

Let 
\[  \sigma_{r+1,n}=\min \{k  : nt_r  \le k \le n-m,  S_k \le S_j \text{ for all } nt_r  \le j \le n-m \}  \]
and for $i \le r$
\[U_i= \tfrac 1{a_n}(S_{\sigma_{i,n}}-S_{nt_{i-1}}) \ , \ V_i=\tfrac 1{a_n}(S_{nt_i}-  S_{nt_{i-1}}) \ , \ V_{r+1}=\tfrac 1{a_n}(S_{n-m}-  S_{nt_{r}})\] and $W_i=(U_i,V_i)$. Since $(a_n)$ is regularly varying, from the last lemma and from our assumptions on independence it follows that the random variables
$Q^{(1)}, \ldots, Q^{(r)}, W_1,\ldots,W_r,V_{r+1} $
are asymptotically independent. Conditioning on the event $\sigma_{r+1,n}=n-m$ does only effect $V_{r+1}$. Thus from Lemma \ref{leQ}
\[ \big(Q^{(1)}, \ldots, Q^{(r)}, W_1, \ldots, W_{r},V_{r+1} \  |\ \sigma_{r+1,n}=n-m\big) \stackrel{d}{\to}  (\mu_k' \otimes \mu_k'')^{\otimes r} \mu_1\otimes \cdots \otimes \mu_r\otimes \nu \ , \]
where the probability measures $\mu_i$ also depend on $t_i-t_{i-1}$. If a Borel set $A \subset \mathbb R^{2r+1}$ satiesfies $\mu_1\otimes \cdots \otimes \mu_r\otimes \nu(A) >0$ and $\mu_1\otimes \cdots \otimes \mu_r\otimes \nu(\partial A) =0$, it follows
\[  \big(Q^{(1)}, \ldots, Q^{(r)} \  |\  (W_1, \ldots, W_{r},V_{r+1}) \in A, \sigma_{r+1,n}=n-m\big) \stackrel{d}{\to}   (\mu_k' \otimes \mu_k'')^{\otimes r} \ .  \]
We apply this result to $A$ of the form $A=B \cap C$, where the Borel set $B$ satisfies the same conditions as $A$, and  \[C= \Big\{(u_1,v_1, \ldots,u_r,v_r,v_{r+1}):u_{j} > \sum_{i=j}^{r+1} v_i\text { for } j \le r \Big\}\ . \] Since $\nu$ is not concentrated on a compact set, $\mu_1\otimes \cdots \otimes \mu_r\otimes \nu(C) >0$, and because $\nu$ has a density, $\mu_1\otimes \cdots \otimes \mu_r\otimes \nu(\partial C) =0$. As
\begin{align*}
\{ \tau_{n-m}=n-m \} =& \ \{  S_{\sigma_{j,n}} > S_{n-m} \text { for } j \le r,\sigma_{r+1,n}=n-m \}\\
=& \  \{ (W_1, \ldots,W_r,V_{r+1}) \in C, \sigma_{r+1,n}=n-m \} 
\end{align*}
we obtain
\[ (Q^{(1)}, \ldots, Q^{(r)} \  |\  (W_1, \ldots, W_{r},V_{r+1}) \in B, \tau_{n-m}=n-m) \stackrel{d}{\to}   (\mu_k' \otimes \mu_k'')^{\otimes r} \ .  \]
The choice $B= \mathbb R^{2r+1}$ gives the asymptotic distribution of $(Q^{(1)}, \ldots, Q^{(r)})$. Since $(S_{\sigma_{1,n}}, S_{nt_1},\ldots,S_{\sigma_{r,n}}, S_{nt_r})$ is obtained from $(W_1, \ldots,W_r,V_{r+1})$ by linear combinations, also the asymptotic independence follows.
\end{proof}

\section{Proof of Theorem 1.1 to Theorem  1.3}

Define
\[ \eta_{k} \ = \ \sum_{y=0}^\infty y(y-1) Q_k(y) \Big/ \Big( \sum_{y=0}^\infty y Q_k(y) \Big)^2 \ , \quad k \geq 1 \ . \]

\begin{lemma} \label{le2} Assume Assumptions A1 to A3. Then for all $x \ge 0$
\[ \sum_{k=0}^\infty \eta_{k+1} e^{-S_k} \ < \ \infty \qquad \mathbf P^+_x \text{ -a.s.} \]
and for all $x \le 0$
\[ \sum_{k=1}^\infty \eta_{k} e^{S_k} \ < \ \infty \qquad \mathbf P^-_x \text{ -a.s.} \]
\end{lemma}

The proof of the first statement can be found in \cite{agkv} (see Lemma 2.7 therein under condition B1 and B2), the second one is proven in just the same way. 

\begin{lemma} \label{le4}
Under Assumptions A1 to A3 there is a non-vanishing finite measure $p'$ on $\mathbb N_0$ with $p'(0)=0$ such that the following holds: Let $Y_n$ be  uniformly bounded random variables of the form $Y_n= \varphi_n(Q_1, \ldots, Q_{n-r_n})$ with natural numbers $r_n \to \infty$  and let $\ell$ be a real number such that for all $m \in \mathbb N_0$
\[\mathbf E[Y_n \mid \tau_{n-m}=n-m] \to \ell \]
as $n \to \infty$. Also let  $\psi: \mathbb N_0 \to \mathbb R$ be a bounded function with $\psi(0)=0$. Then
\[  \mathbf E[Y_n \psi(Z_n) e^{-S_n} \mid \tau_n=n ] \to  \ell \int \psi \, dp' \ . \]
\end{lemma}

\begin{proof} 
Let $f_n(s)= \sum_{ k \ge 0} s^k Q_n(k) $, $0 \le s \le 1$, be the (random) generating function of $Q_n$, $n \ge 1$, and denote
\[f_{j,k} = \begin{cases} f_{j+1}\circ f_{j+2} \circ \cdots \circ f_k & \text{for } 0 \le j < k\ , \\
\text{id}  & \text{for }j=k\ , \\
f_{j}\circ f_{j-1} \circ \cdots \circ f_{k+1}  & \text{for } 0 \le k<j\ .
\end{cases}
\]
 For $0 \le k < n$ set
\[ L_{k,n}= \min(S_{k+1}, \ldots,S_n)-S_k \text{ and } L_{n,n}=0 \ . \]

First we look at the case $\psi(z)= 1-s^z$ with $0 \le s < 1$ (with $0^0=1$). We decompose the expectation according to the value of $\tau_{n-m}$ for some fixed $m \in \mathbb N_0$. For convenience we assume $0 \le Y_n \le 1$. Then for $l>m$ because of $\mathbf E[Z_n\mid \Pi]= e^{S_n}$ a.s. and $1-s^z \le z$
\begin{align*}
\mathbf E \big[ &Y_n  (1-s^{Z_n} )e^{-S_n}; \tau_{n-m} < n-l, \tau_n=n \big] \\ &\le 
\mathbf E \big[Z_n e^{-S_n}; \tau_{n-m} < n-l, \tau_n=n \big]=\mathbf P(\tau_{n-m} < n-l,\tau_n=n)\ .
\end{align*}
From duality 
\[ \mathbf P(   \tau_{n-m}< n-l, \tau_n=n)  \le \mathbf P(  S_k \ge S_m \text{ for some } l<k \le n, M_n < 0) \] 
and in view of Lemma \ref{limitEminus}
\begin{align*}
 \mathbf P(&S_k \ge S_m \text{ for some } l < k \le n, M_n < 0) \\ &\sim \mathbf P^-(S_k \ge S_m \text{ for some } k > l) \mathbf P(M_n < 0) \ .
 \end{align*}
 Since $S_k \to - \infty$ $\mathbf P^-$-a.s. (see Lemma 2.6 in \cite{agkv}), we obtain that for given $\varepsilon>0$ and $m\in \mathbb N$ the estimate $\mathbf P^-(S_k \ge S_m \text{ for some } k > l) < \varepsilon$ is valid, if only $l$   is chosen large enough. Altogether this implies that for $l$ sufficiently large 
\begin{align*}
\mathbf E[&Y_n ( 1-s^{Z_n}) e^{-S_n}; \tau_n=n ]\\ &=  \mathbf E[Y_n (1-s^{Z_n}) e^{-S_n};  \tau_{n-m}\ge n-l, \tau_n=n ] + \chi_1
\end{align*} 
where $|\chi_1| \le \varepsilon \mathbf P(\tau_n=n)$. 

Next from the branching property
\begin{align*}\mathbf E[Y_n &(1-s^{Z_n}) e^{-S_n};\tau_{n-m}\ge n-l, \tau_n=n ] \\ &= \mathbf E[Y_n (1-f_{0,n}(s)) e^{-S_n}; \tau_{n-m}\ge n-l,\tau_n=n ] \ .
\end{align*}
By means of duality
\begin{align*}
\big| \mathbf E[&Y_n ( 1-f_{0,n}(s)) e^{-S_n} ; \tau_{n-m}\ge n-l, \tau_n=n ] \\\ &\quad \mbox{}  - \mathbf E[Y_n (1-f_{n-m,n}(s)) e^{-(S_n-S_{n-m})}; \tau_{n-m}\ge n-l,\tau_n=n ] \big|\\ &\le \mathbf E\big[\big| ( 1-f_{0,n}(s)) e^{-S_n} -(1-f_{n-m,n}(s)) e^{-(S_n-S_{n-m})}\big|\ ; \tau_n=n \big]\\
&= \mathbf E\big[\big| ( 1-f_{n,0}(s)) e^{-S_n} -(1-f_{m,0}(s)) e^{-S_m }\big|\ ; M_n < 0 \big] \ .
\end{align*}
Now $U_n(s)= (1-f_{n,0}(s)) e^{-S_n}$ is decreasing in $n$ (see Lemma 2.3 in \cite{gkv}) with limit $U_\infty(s)$, and for given $\varepsilon > 0$ we obtain from Lemma \ref{limitEminus} for $n$ large enough
\begin{align*}
\big|& \mathbf E [Y_n ( 1-f_{0,n}(s)) e^{-S_n} ; \tau_{n-m}\ge n-l,\tau_n=n ] \\\ &\quad \mbox{}  - \mathbf E[Y_n (1-f_{n-m,n}(s)) e^{-(S_n-S_{n-m})}; \tau_{n-m}\ge n-l, \tau_n=n ] \big|\\ &\qquad\qquad\qquad\qquad \le 2\mathbf E^- [U_m(s)- U_\infty(s)]\mathbf P(\tau_n=n) \le   \varepsilon  \mathbf P(\tau_n=n)\ ,
\end{align*}
if only $m$ is chosen large enough. Now $\{\tau_{n-m}\ge n-l, \tau_n=n\}$ may be decomposed as $ \bigcup_{j=m}^l \big(\{\tau_{n-j}=n-j\}\cup\{L_{n-j,n-m} \ge 0, \tau_n=n\}\big)$
and for  large $n$ by $Y_n= \varphi_n(Q_1,\ldots,Q_{n-r_n})$
\begin{align*}
\mathbf E[Y_n (&1-f_{n-m,n}(s)) e^{-(S_n-S_{n-m})}; \tau_{n-j}=n-j, L_{n-j, n-m} \ge 0,\tau_n=n ] \\
&= \mathbf E[Y_n; \tau_{n-j}=n-j] \mathbf E[(1- f_{j-m,j}(s))e^{-(S_j-S_{j-m})};L_{j-m} \ge 0, \tau_j=j] \ .
\end{align*}
By assumption $\mathbf E[Y_n; \tau_{n-j}=n-j]  \sim \ell \mathbf P(\tau_n=n)$.
Putting pieces together we obtain
\begin{align*}
\mathbf E[&Y_n ( 1-s^{Z_n}) e^{-S_n}; \tau_n=n ]\\ &=  \mathbf E[Y_n (1-s^{Z_n}) e^{-S_n};  \tau_{n-m}\ge n-l, \tau_n=n ] + \chi_1 \\
&=  \ell\, \mathbf P(\tau_n=n)\sum_{j=m}^{ l} \mathbf E[(1- f_{j-m,j}(s))e^{-(S_j-S_{j-m})};L_{j-m} \ge 0, \tau_j=j] + \chi_2
\end{align*}
where $|\chi_2| \le 3\varepsilon \mathbf P(\tau_n=n)$. In particular we may apply this formula for $Y_n=1$, to obtain for large $n$
$$\big|\mathbf E[ Y_n ( 1-s^{Z_n}) e^{-S_n}; \tau_n=n ]- \ell \mathbf E[ ( 1-s^{Z_n}) e^{-S_n}; \tau_n=n ] \big| \le 6\varepsilon \mathbf P(\tau_n=n)$$
and our computations boil down to the formula
\begin{align*}
\mathbf E[Y_n ( 1-s^{Z_n}) e^{-S_n}; \tau_n=n ] \sim  \ell\, \mathbf E[  ( 1-s^{Z_n}) e^{-S_n}; \tau_n=n ] \ .
\end{align*}
The right-hand side may be written as $\ell\, \mathbf E[( 1-f_{n,0} (s)) e^{-S_n}; M_n<0 ]$ and another application of Lemma \ref{limitEminus} gives altogether
\[ \mathbf E[ Y_n ( 1-s^{Z_n}) e^{-S_n}; \tau_n=n] \sim \ell\, \mathbf E^- [U_\infty (s)]\mathbf P(\tau_n=n)\ . \]

In view of $s^z1_{z>0}= (1-0^z)-(1-s^z)$ this implies
\begin{align}
\mathbf E \big[Y_n s^{Z_n} e^{-S_n}; Z_n>0,\tau_n=n \big] \sim \ell\, h(s)\mathbf P(\tau_n=n) \label{continuity}
\end{align}
with $h(s)= \mathbf E^- [U_\infty(0)-U_\infty(s)]$. 

Now we show that $h(1)=\mathbf E^-[U_\infty(0)]>0$. This follows from  an estimate due to Agresti (see \cite{ag} and the proof of Proposition 3.1 in \cite{agkv}), which in our case reads
\[ (1-f_{k,0}(s))e^{-S_k} \ge \Big( \frac 1{1-s} + \sum_{i=1}^k \eta_ie^{S_i}\Big)^{-1} \ . \]
Letting $k\to \infty$ Lemma \ref{le2} implies $U_\infty(s) >0$ $\mathbf E^-$-a.s. and thus $h(s)>0$ for all $s<1$. For $s=0$ it follows that $h(1)=\mathbf E^-[U_\infty(0)]>0$. 

Also from  $\mathbf E [   Z_n e^{-S_n};  \tau_n=n ]= \mathbf E [  \mathbf E[Z_n\mid \Pi] e^{-S_n};  \tau_n=n ]=\mathbf P(\tau_n=n)$
and from $1-s^z \le z(1-s)$ we get
\begin{align*} \mathbf E[U_n(s)&; M_n < 0] =\mathbf E \big[(1-s^{Z_n}) e^{-S_n} ;  \tau_n=n \big]  \\ &\le (1-s) \mathbf E \big[Z_n e^{-S_n}; \tau_n=n \big]  = (1-s) \mathbf P(\tau_n=n)
\end{align*}
which for $n \to \infty$ implies $h(1)-h(s)=\mathbf E^-[U_\infty(s)] \le 1-s$. Therefore $h$ is continuous at $s=1$.
Our claim follows now from \eqref{continuity} and the continuity theorem for generating functions. 
\end{proof}

\begin{lemma} \label{le3}
Let $Y_n$ fulfil the same conditions as in Lemma \ref{le4}. Then under Assumptions A1 to A3 there is  a non-vanishing finite measure $p''$ on $\mathbb N\times \mathbb N_0$  such that for every bounded $\psi:\mathbb N \times \mathbb N_0 \to \mathbb R$
\[ \frac{\mathbb E[ Y_n \psi(Z_n,n-\tau_n);Z_n>0]}{\gamma^n \mathbf P(\tau_n=n)} \to   \ell \int \psi \, dp'' \ . \]
\end{lemma}

\begin{proof} We have for fixed $j \in \mathbb N_0$
\begin{align*}
\gamma^{-n} \mathbb E[  Y_n \psi(Z_n,n-\tau_n)&;Z_n>0, \tau_n=n-j] \\&= \mathbf E[Y_n \psi_j(Z_{n-j})e^{-S_{n-j}}; \tau_{n-j}=n-j]
\end{align*}
with $ \psi_j(z) =  \mathbf E[ \psi(Z_j,j)e^{-S_j};Z_j>0, L_j \ge 0 \mid Z_0=z]$ for $z >0$ and $\psi_j(0)=0$. Also there is a finite measure $p_j'$ such that $\int \psi_j \, dp' = \int \psi(\, \cdot \, , j) \, dp_j'$. From the preceding lemma
\begin{align*}
\frac{\mathbb E[  Y_n \psi(Z_n,n-\tau_n);Z_n>0, \tau_n=n-j] }{\gamma^n \mathbf P(\tau_n=n)} \to \ell \int \psi(\, \cdot \, , j) \, dp_j' \ . 
\end{align*}
In particular $p_0'$ is non-vanishing.
Thus it remains to show that for given $\varepsilon >0$ there is a natural number $k$ such that
\[ \gamma^{-n} \mathbb E[  Y_n \psi(Z_n,n-\tau_n) ;Z_n>0, \tau_n\le n-k] \le \varepsilon \mathbf P(\tau_n=n) \]
for large $n$. Without loss $0 \le Y_n \le 1$ and $0 \le \psi \le 1$. Then
\begin{align*}
\gamma^{-n} \mathbb E[  Y_n &\psi(Z_n,n-\tau_n) ;Z_n>0, \tau_n\le n-k] \le \mathbf E[    e^{-S_n};Z_n>0 ,\tau_n \le n-k] \\
&\le \sum_{i=0}^{n-k}  \mathbf E[ e^{-S_n};Z_i>0,\tau_i=i, L_{i,n}\ge 0] \\ 
&\le 
\sum_{i=0}^{n-k}  \mathbf E[  e^{S_i-S_n}; \tau_i=i, L_{i,n}\ge 0] 
= 
\sum_{i=0}^{n-k}  \mathbf P(    \tau_i=i) \mathbf E [e^{-S_{n-i}}; L_{n-i}\ge 0]\ .
\end{align*}
From  Lemmas \ref{pro1}, \ref{maxim} both $\mathbf P(   \tau_{n}=n)$ and $\mathbf E[e^{-S_{n}};L_{n} \ge 0]$ are regularly varying with negative indices. Therefore for large $n$
\begin{align*}
\gamma^{-n} \mathbb E[  Y_n &\psi(Z_n,n-\tau_n) ;Z_n>0, \tau_n\le n-k]\\ & \le \mathbf E [e^{-S_{  n/3  }}; L_{  n/3  }\ge 0]\sum_{i\le n/2}  \mathbf P(    \tau_i=i) \\ &\qquad\qquad \mbox{} + \mathbf P(    \tau_{  n/3  }=  n/3   )\sum_{k\le j \le n/2}   \mathbf E [e^{-S_{j}}; L_{j}\ge 0]\ .
\end{align*}
Also $ \mathbf E[e^{-S_n}; L_n \ge 0]= o(\tfrac 1n)$ and  $\sum_{i\le n}  \mathbf P(    \tau_i=i)= O(n \mathbf P(\tau_n=n))$  and $\sum_{j \ge 1}   \mathbf E [e^{-S_{j}}; L_{j}\ge 0]< \infty$. Therefore for every $\varepsilon > 0$ the right-hand side of the inequality above  is bounded by $\varepsilon \mathbf P(\tau_n=n)$, if $k$ is large enough. This gives the claim.
\end{proof}

\noindent
Choosing $Y_n=1$ and $\psi = 1_{\mathbb N\times \mathbb N_0}$, we obtain Theorem \ref{survival}. 

\begin{proof}[Proof of Theorem \ref{theoZ_n}] In view of Theorem \ref{survival}, the first part is a special case of Lemma \ref{le3} with $Y_n=1$ and $\psi(Z_n,n-\tau_n)=1-s^{Z_n}$. For the second part we use H\"older's inequality (with $1/p=\beta, 1/q=1-\beta$) and \eqref{gleichung1}
\begin{align*}
\gamma^{-n} \mathbb E [ Z_n^\beta] &=\mathbf E[\mathbf E[Z_n^\beta 1_{Z_n >0} \mid \Pi] e^{-S_n} ]\\
&\le \mathbf E[ \mathbf E[Z_n\mid \Pi]^\beta \mathbf P(Z_n>0 \mid\Pi)^{1-\beta}e^{-S_n}] \le \mathbf E[e^{(1-\beta)(L_n-S_n)}]\ .
\end{align*}
Again we decompose with $\tau_n$ and obtain
\begin{align*}
\gamma^{-n} \mathbb E [ Z_n^\beta] &\le \sum_{i=0}^n \mathbf E[e^{(1-\beta)(L_n-S_n)}; \tau_i=i, L_{i,n} \ge 0]\\ & = \sum_{i=0}^n \mathbf P(\tau_i=i) \mathbf E[e^{-(1-\beta)S_{n-i}}; L_{n-i} \ge 0] \ .
\end{align*}
As above we show by means of Lemma \ref{pro1} with $r=1-\beta$ that this quantity is of order $\mathbf P(\tau_n=n)$, and the claim follows.
\end{proof}
\begin{proof}[Proof of Theorem  \ref{limitlaw}]
Again the first part is a special case of Lemma \ref{le3}. Next
let $\varphi: D[0,1]\to \mathbb R$ be bounded and continuous. We apply Lemma \ref{le3} to $Y_n=\varphi( \tfrac 1{a_n} \bar S^{n})$, where $\bar S^n_t = S_{  nt   \wedge r_n}$  with  natural numbers $r_n \to \infty$. If $n-r_n=o(n)$, then it follows from Lemma \ref{funclimit} and standard arguments that $\mathbf E[Y_n \mid \tau_{n-m}=n-m] \to \mathbf E [\varphi(L^*)]$. Lemma \ref{le3} yields 
\[ \mathbb E [ \varphi( \tfrac 1{a_n} \bar S^{n}) \mid Z_n >0] \to \mathbf E [\varphi(L^*)] \ . \]
Thus $(\tfrac 1{a_n} \bar S^{n} \mid Z_n >0) \stackrel{d}{\to} L^*$. Also conditional asymptotic independence follows from Lemma \ref{le3}. Finally for fixed $r$
\begin{align*}
 \mathbb P(&|X_{n-r+1}|+ \cdots + |X_n| \ge \sqrt{a_n} ; Z_n > 0) \\ &\le  \mathbb P(Z_{n-r} > 0) \mathbb P(|X_{1}|+ \cdots + |X_r| \ge \sqrt{a_n} ) =o(\mathbb P(Z_n>0))\ .
 \end{align*}
This holds true also, if $r=r_n\to \infty$ sufficiently slow. It follows
\[ \gamma^{-n} \mathbb P( \tfrac 1{a_n} \sup|S^n-\bar S^n| \ge \varepsilon \mid Z_n>0) \to 0 \]
for all $\varepsilon >0$. Therefore $( \tfrac 1{a_n} (S^n-\bar S^n) \mid Z_n >0) \stackrel{d}{\to} 0$ in $D[0,1]$ and consequently $(\tfrac 1{a_n} S^{n} \mid Z_n >0) \stackrel{d}{\to} L^*$. This finishes the proof.
\end{proof}

\section{Trees with stem}

For every $n=0,1,\ldots, \infty$ let $\mathcal T_n$ be the set of all ordered rooted trees of height exactly $n$. For a precise definition we refer to the coding of ordered trees and their nodes given by Neveu \cite{ne}. Then $\mathcal T_{\ge n}=\mathcal T_n\cup \mathcal T_{n+1} \cup \cdots \cup \mathcal T_\infty$ is the set of ordered rooted trees of  at least height $n$. With $[\ ]_{n}: \ \mathcal T_{\ge n}\to \mathcal T_n$ we denote the operation of pruning a tree $t \in \mathcal T_{\ge n}$ to a tree $[t]_n\in \mathcal T_n$ of height exactly $n$ by eliminating all nodes of larger height.

For $n=0,1,\ldots, \infty$ a tree with a stem of height $n$, shortly a {\em trest} of height $n$, is a pair  
\[ \mathsf{t} = (t, k_0 k_1\ldots k_n) \ , \]
where $t \in \mathcal T_{\ge n}$ and $k_0, \ldots,k_n$ are nodes in $t$ such that $k_0$ is the root (founding ancestor) and $k_{i}$ is an offspring of $k_{i-1}$. Thus $k_i$ belongs to generation $i$. We call $k_0 \ldots k_n$ the stem within $\mathsf t$,  it is determined by $k_n$. $\mathcal T_n'$ denotes the set of all trests of height $n$.

A trest $\mathsf{t} = (t, k_0 k_1\ldots k_n) $ of height $n$ can also be pruned at height $m \le n$ to obtain 
the trest of height $m$
\[ [\mathsf{t}]_m = ([t]_m,k_0 \ldots k_m) \ .\]

To every tree $t \in \mathcal T_{\ge n}$ there belongs a unique trest
\[ \langle t \rangle_n = ([t]_n, k_0(t) \ldots k_n(t))  \]
of height $n$, where $k_0(t) \ldots k_n(t)$ is the {\em leftmost} stem, which can be fitted into $[t]_n$. Notice that this stem is uniquely determined, since $t$ is ordered and of at least height $n$.\\\\
Now let $\pi= (q_1,q_2, \ldots)$ be a fixed environment. Define the distribution $\tilde q_i$ by its weights
\[ \tilde q_i(y) = \tfrac 1{m(q_i)} yq_i(y) \ , \quad y=0,1,\ldots   \]
Then a corresponding {\em LPP-trest} (Lyons-Pemantle-Peres trest) is the random trest $\tilde{\mathsf T}= (\tilde T, \tilde K_0\tilde K_1 \ldots)$ with values in $\mathcal T_\infty'$ satisfying the following properties:\\
Given $\Pi= (q_1,q_2, \ldots)$  a.s. 
\begin{itemize}
\item
the offspring numbers of all individuals are independent random variables,
\item
the offspring number of $\tilde K_{i-1}$ has distribution $\tilde q_{i}$ and the offspring number of any other individual in generation $i-1$ has distribution $q_{i}$, and
\item 
the node $\tilde K_{i}$ is uniformly distributed among all children of $\tilde K_{i-1}$, given the offspring number of $\tilde K_{i-1}$ and given all other random quantities.
\end{itemize}
Shortly speaking: From the infinite stem individuals grow  according to a size biased distribution, and from the other individuals ordinary branching trees arise to the right and left of the stem. Such type of trests have been considered by Lyons, Peres and Pemantle \cite{lpp} in the Galton-Watson case.

Let $\tilde Z_n$ be the  population size of the LPP-trest in generation $n$. 

\begin{lemma}\label{le41}
Under Assumptions A1 to A3
\[ e^{-S_n} \tilde Z_n \to W^+  \qquad \mathbf P^+\text{-a.s.} \]
with some random variable $W^+$ fulfilling $W^+ >0$ $\mathbf P^+$-a.s.
\end{lemma}

\begin{proof}
We use the representation
\[ \tilde Z_n = 1 + \sum_{i=0}^{n-1} \tilde Z_n^i \]
where $\tilde Z^i_n$ is the number of individuals in generation $n$ other than $\tilde K_{n}$, which descent from $\tilde K_i$ but not from $\tilde K_{i+1}$. Thus
$\mathbf E[ \tilde Z^i_{i+1} \mid \Pi] = \sum_y y\tilde Q_{i+1} (y) -1 = e^{X_{i+1}} \eta_{i+1} $ and a.s.
\begin{align} \mathbf E[ \tilde Z^i_{n} \mid \Pi] = e^{S_n-S_{i+1}} \mathbf E[ \tilde Z^i_{i+1} \mid \Pi]=  \eta_{i+1} e^{S_n-S_i} \ . \label{Ztilde}
\end{align}

Now given the environment
$ e^{-S_n} \sum_{i=k}^{n-1}\tilde Z_n^i $
is for $n>k$ a non-negative submartingale. Therefore Doob's inequality implies
that for every $\varepsilon \in (0,1)$
\[ \mathbf P\Big( \max_{k<m\le n} e^{-S_m} \sum_{i=k}^{m-1} \tilde Z_m^i \ge \varepsilon\ \Big|\ \Pi\Big) \le \frac 1\varepsilon    \sum_{i=k}^{n-1}  e^{-S_n}\mathbf E[\tilde Z_n^i \mid \Pi] \le \frac 1\varepsilon    \sum_{i\ge k} \mathbf \eta_{i+1} e^{-S_i} \]
and
\[ \mathbf P^+\Big(\sup_{m>k} e^{-S_m} \sum_{i=k }^{m-1} \tilde Z_m^i \ge \varepsilon \Big) \le \frac 1\varepsilon \mathbf E^+ \Big[1\wedge \sum_{i\ge k} \mathbf \eta_{i+1} e^{-S_i} \Big] \ .\]
From Lemmas \ref{limitEminus} and \ref{le2} it follows that
\[ \mathbf P^+\Big(\sup_{m>k} e^{-S_m} \sum_{i=k }^{m-1} \tilde Z_m^i \ge \varepsilon \Big) \le \varepsilon\ , \]
if $k$ is chosen large enough.
Also $e^{-S_n} \tilde Z_n^i$ is for $n \ge i+1$ and a fixed environment a non-negative martingale, such that for $n \to \infty$ 
\[ e^{-S_n} \tilde Z_n^i \to W^i \qquad \mathbf P^+\text{-a.s.}\]
These facts together with $ S_n \to \infty$ $\mathbf P^+$-a.s. imply that
\[ e^{-S_n} \tilde Z_n \to W^+ \qquad \mathbf P^+\text{-a.s.} \]
for some random variable $W^+$. Also $W^+ \ge \sum_{i\ge 0} W^i$ $\mathbf P^+$-a.s. 

Thus it remains to show that $\sum_{i\ge 0} W^i > 0 $ $\mathbf P^+$-a.s. Given $\Pi$, the random variables $W^i$ are independent, since they arise from independent branching processes in the LPP-trest. In view of the second Borel-Cantelli Lemma it is thus sufficient to prove
\[ \sum_{i\ge 0} \mathbf P^+ (W^i >0 \mid \Pi) = \infty \qquad \mathbf P^+\text{-a.s.} \]
Now we use the formula
\[ \mathbf P^+ (W^i >0 \mid \Pi) \ge \Big( \sum_{j=i}^\infty \eta_{j+1} e^{-(S_j-S_i)} \Big)^{-1} \ , \]
which is taken from the proof of Proposition 3.1 in \cite{agkv} (a few lines after (3.7) therein). Because of Lemma \ref{le2} above the right-hand side is strictly positive $\mathbf P^+$-a.s.  Moreover there are random times $0= \nu(0)<\nu(1) < \cdots$ such that 
\[ \Big( \sum_{j=\nu(k)}^\infty \eta_{j+1} e^{-(S_j-S_{\nu(k)})} \Big)^{-1} \ , \quad k=0,1, \ldots \]
is a stationary sequence of random variables, which is a consequence of Tanaka's decomposition, see \cite{ta} and Lemma 2.6 in \cite{agkv}. From Birkhoff's ergodic theorem it follows that
\[\frac 1n \sum_{k=1}^n \Big( \sum_{j=\nu(k)}^\infty \eta_{j+1} e^{-(S_j-S_{\nu(k)})} \Big)^{-1}  \]
has a strictly positive limit $\mathbf P^+$-a.s. This implies our claim.
\end{proof}

\noindent
We use the LPP-tree to approximate conditioned BPRE. Let
us denote by $T$  a branching tree in random environment $\Pi$.  This is nothing else than the ordered rooted tree belonging to a BPRE in environment $\Pi$. Again let $Z_n$ denote its number of individuals in generation $n$.

\begin{theorem} \label{trest} 
Assume A1 to A3. Let $0 \le r_n< n $ be a  sequence of natural numbers with $r_n \to \infty$. Let $Y_n$ be uniformly bounded random variables of the form $Y_n=\varphi(Q_1,\ldots,Q_{n-r_n})$  and let $B_n \subset \mathcal T_{n-r_n}'$, $n \ge 1$. If for some $\ell \ge 0$
\[ \mathbf E\big[Y_n; [\tilde {\mathsf T}]_{n-r_n} \in B_n \ \big| \ \tau_{n-m}= n-m\big] \to \ell \]
for  all $m \ge 0$, then
\[ \mathbb E\big[Y_n;[\langle T\rangle_n]_{n-r_n} \in B_n \ \big| \ Z_n > 0\big] \to \ell  \ . \]
$B_n$ may be random, depending only on the environment $\Pi$.
\end{theorem}
\noindent
For the proof we use the following theorem due to J. Geiger (see \cite{ge_99}). Let $\pi=(q_1,q_2,\ldots)$ be a fixed environment, let $\mathbf P_\pi(\cdot)$ be the corresponding probabilities and let 
\[ \mathsf T_{n,\pi}= (T_n, K_0 \ldots K_n) \]
denote a random trest of height $n$ and let for $i=1, \ldots, n$
\begin{align*} 
T_i' &= \text{subtree within } T_n \text{ right to the stem with root } K_{i-1} \ , \\ 
T_i'' &= \text{subtree within } T_n \text{ left to the stem with root } K_{i-1} \ , \\
R_i &= \text{size of the first generation of } T_i' \ ,  \\
L_i &= \text{size of the first generation of } T_i''   \ .
\end{align*}
For $\mathsf T_{n,\pi}$ the following properties are required:
\begin{itemize}
\item
$\mathbf P_\pi ( R_i=r, L_i=l) = q_i(r+l+1) \frac{\mathbf P_\pi(Z_n >0 \ | Z_i=1) \mathbf P_\pi(Z_n=0 | Z_i=1 )^{l}}{\mathbf P_\pi(Z_n >0 \ | Z_{i-1}=1)}$.
\item
 $T_i'$, if decomposed at its first generation, consists of $R_i$ subtrees $\tau_{ij}'$, $j=1, \ldots, R_i$, which are branching trees within the fixed environment $(q_{i+1},q_{i+2}, \ldots)$.
\item
Similarly  $T_i''$ consists of $L_i$ subtrees $\tau_{ij}''$, which are branching trees within the fixed environment $(q_{i+1},q_{i+2}, \ldots)$  conditioned to be  extinct before generation $n-i$.
\item
All pairs $(R_i,L_i)$  and all subtrees $\tau_{ij}'$, $\tau_{ij}''$  are independent.
\end{itemize}
These properties determine the distribution of $\mathsf T_{n,\pi}$ up to possible offspring of $K_n$ and thus the distribution of $\langle \mathsf T_{n, \pi}\rangle_n$. 
\begin{theorem} \label{geigerconstruction} 
For almost all $\pi$ the conditional distribution of $\langle T\rangle_n$ given  $\Pi=\pi,Z_n>0$ is equal to the distribution of $\langle \mathsf T_{n, \pi}\rangle_n$.
\end{theorem}
\noindent
Geiger  proved this result for a fixed environment $q_1=q_2 = \cdots$ i.e. in the Galton-Watson case, see Proposition 2.1 in \cite{ge_99}. His proof carries over straightforward to a varying environment.

\begin{proof}[Proof of Theorem \ref{trest}]
\mbox{}\\
For the trest $\tilde{\mathsf T}$ we introduce the notations $\tilde T_i', \tilde T_i'', \tilde R_i, \tilde L_i, \tilde \tau_{ij}', \tilde \tau_{ij}''$. They have the same meaning as above $  T_i',   T_i'',   R_i,   L_i,   \tau_{ij}',   \tau_{ij}''$ for the trest $\mathsf T_{n,\pi}$. From the construction of $\tilde{\mathsf T}$
\begin{align*}
\mathbf P_\pi (\tilde  R_i=r, \tilde L_i=l) = q_i(r+l+1) e^{-X_i} \ .
\end{align*}
$\tilde \tau_{ij}'$ and  $\tau_{ij}'$ are equal in distribution, whereas $\tilde \tau_{ij}''$ is no longer conditioned to be extinct in generation $n-i$, as this is the case for $\tau_{ij}''$. 

In order to compare both trests we will couple them.
We first consider the branching process in a fixed environment $\pi=(q_1,q_2, \ldots)$ and again write the corresponding probabilities as $\mathbf P_\pi(\cdot)$. To begin with we estimate the total variation distance between the distributions of $(R_i,L_i)$ and $(\tilde R_i, \tilde L_i)$. Note that
\begin{align*}
\mathbf P_\pi ( Z_n > 0\mid Z_{i-1}=1 )  &= \sum_{j \ge 1} \mathbf P_\pi( Z_n >0\mid Z_{i }=j) \mathbf P_\pi( Z_i=j \mid Z_{i-1}=1) \\
&\le \sum_{j \ge 1} j \mathbf P_\pi( Z_n>0 \mid Z_{i}=1)\mathbf P_\pi( Z_i=j \mid Z_{i-1}=1) \\ &= e^{X_i} \mathbf P_\pi( Z_n >0\mid Z_{i }=1) 
\end{align*}
such that for   $r,l,m \ge 0$ and $i \le n-m$
\begin{align*}
\mathbf P_\pi (\tilde R_i=r, \tilde L_i=l&)- \mathbf P_\pi (  R_i=r,   L_i=l)\\
&\le q_i(r+l+1) e^{-X_i} \big(1- \mathbf P_\pi(Z_n=0 \mid Z_{i}=1)^{l}\big)\\&\le q_i(r+l+1) e^{-X_i} l\big(1- \mathbf P_\pi(Z_n=0 \mid Z_{i}=1) \big)\\ 
&\le lq_i(r+l+1)  e^{-X_i} \mathbf P_\pi(Z_{n-m}>0\mid Z_i=1) \\
&\le lq_i(r+l+1) e^{-X_i} e^{S_{n-m}-S_i} \ .
\end{align*}
Since the right-hand side is always non-negative, we may estimate the total variation distance as 
\begin{align*}\tfrac 12 \sum_{r,l \ge 0}  \big|\mathbf P_\pi &(\tilde R_i=r, \tilde L_i=l)- \mathbf P_\pi  (  R_i=r,   L_i=l)\big|
\\&= \sum_{r,l \ge 0}  \big(\mathbf P_\pi (\tilde R_i=r, \tilde L_i=l)- \mathbf P_\pi  (  R_i=r,   L_i=l)\big)^+
 \\ &\le e^{-X_i}e^{S_{n-m}-S_i} \sum_{r,l\ge 0} lq_i(r+l+1) \\ &= e^{-X_i}e^{S_{n-m}-S_i}\tfrac 12  \sum_{y=1}^\infty y(y-1) q_i(y) = \tfrac 12 \eta_i e^{S_{n-m}-S_{i-1}} \ .
\end{align*}

Similarly we estimate the total variation distance between the distributions of $\tau_{ij}''$ and $\tilde \tau_{ij}''$.  The second distribution is equal to the first distribution conditioned to be extinct in generation $n-i$. This event can be expressed as $\{\tau_{ij}'' \in B_i\}$ with the set $B_i$ of trees of height less than $n-i$, thus for some tree $t$
\begin{align*}
\mathbf P_\pi( \tau_{ij}''=t) - \mathbf P_\pi( \tilde \tau_{ij}''=t) &= \mathbf P_\pi( \tau_{ij}''=t) - \mathbf P_\pi( \tau_{ij}''=t\mid \tau_{ij}'' \in B_i) \\ &\le \mathbf P_\pi(\tau_{ij}''=t) 1_{B_i^c}(t) \ .
\end{align*}
Again, since the right-hand side is non-negative for $i \le n-m$
\begin{align*}
\tfrac 12 \sum_t \big|&\mathbf P_\pi(  \tau_{ij}''=t) - \mathbf P_\pi( \tilde \tau_{ij}''=t)\big|  \le  \mathbf P_\pi (\tau_{ij}''\in B_i^c) \\ & = \mathbf P_\pi(Z_{n} > 0\mid Z_i=1) \le \mathbf P_\pi(Z_{n-m} > 0\mid Z_i=1) \le e^{S_{n-m}-S_i} \ .
\end{align*}

Now we consider the following construction: Take couplings of the pairs $(R_i,L_i )$, $(\tilde R_i, \tilde L_i)$ and of $\tau_{ij}''$ and $\tilde \tau_{ij}''$. Also let $\tau_{ij}'=\tilde\tau_{ij}'$. Put these components together to obtain $(T_i',T_i'')$ and $(\tilde T_i', \tilde T_i'')$. If the couplings are all independent of each other, then the   resulting trests have the required distributional properties. We denote the resulting probabilities again by $\mathbf P_\pi$. Thus
\begin{align*}
\mathbf P_\pi(&(T_i',T_i'')\neq (\tilde T_i', \tilde T_i'')   ) \\ &\le  \mathbf P_\pi((R_i,L_i )\neq (\tilde R_i, \tilde L_i)) + \sum_{r,l \ge 0}\sum_{j=1}^l \mathbf P_\pi(\tilde R_i=r,\tilde L_i=l) \mathbf P_\pi (\tau_{ij}''\neq \tilde \tau_{ij}'')\ .
\end{align*}
For optimal couplings we may use the above estimates on the total variation distance
and obtain for $i \le n-m$
\begin{align*}
\mathbf P_\pi( (T_i',T_i'')\neq (\tilde T_i', \tilde T_i'')   ) &\le \tfrac 12 \eta_i e^{S_{n-m}-S_{i-1}} + \sum_{r,l\ge 0} l q_i(r+l+1)   e^{-X_i} e^{S_{n-m}-S_i}\\ &=   \eta_i e^{S_{n-m}-S_{i-1}} \ .
\end{align*}
Altogether using Theorem \ref{geigerconstruction} and the assumption that $B_n$ depends only on $\Pi$, it follows for $m < r_n$
\[ \big| \mathbf P_\pi \big([\langle T\rangle_n]_{n-r_n} \in B_n \ \big| \ Z_n > 0\big) - \mathbf P_\pi \big( [\tilde {\mathsf T}]_{n-r_n} \in B_n \big) \big|\le 1 \wedge\sum_{i=1}^{n-r_n} \eta_i e^{S_{n-m}-S_{i-1}} \ . \]
Now from duality and from Lemmas \ref{limitEminus}, \ref{le2}
\begin{align*}
\mathbf E\Big[1 \wedge\sum_{i=1}^{n-r_n} \eta_i &e^{S_{n-m}-S_{i-1}} \mid \tau_{n-m}=n-m\Big] \\ &=
\mathbf E\Big[1 \wedge \sum_{i=r_n-m}^{n-m} \eta_i e^{S_i} \mid M_{n-m} < 0\Big] \to 0  \ . 
\end{align*}
According to our assumptions $ \mathbf E[Y_n\mathbf P_\Pi \big( [\tilde {\mathsf T}]_{n-r_n} \in B_n \big) \mid \tau_{n-m}=n-m]$ converges to $\ell$. Our estimates thus imply that
\[ \mathbf E[Y_n \mathbf P_\Pi \big([\langle T\rangle_n]_{n-r_n} \in B_n \ \big| \ Z_n > 0\big) \mid \tau_{n-m}=n-m)] \to \ell \ . \]
Thus we may apply Lemma \ref{le3} with $Y_n \mathbf P_\Pi( \big([\langle T\rangle_n]_{n-r_n} \in B_n \mid Z_n > 0\big)$ instead of $Y_n$, $\psi= 1$  to obtain
\[ \frac{\mathbf E\big[Y_n\mathbf P_\Pi \big([\langle T\rangle_n]_{n-r_n} \in B_n \mid Z_n > 0\big) ; Z_n>0\big]}{\gamma^n \mathbf P(\tau_n=n)} \to \ell p''(\mathbb N \times \mathbb N_0)\ . \]
Also from Lemma \ref{le3} with $Y_n=1$ and $\psi=1$ 
\[ \mathbb P(Z_n>0) \sim \gamma^n \mathbf P(\tau_n=n) p''(\mathbb N \times \mathbb N_0) \ , \]
thus
\[ \mathbf E\big[Y_n\mathbf P\big([\langle T\rangle_n]_{n-r_n} \in B_n \mid \Pi ,Z_n > 0\big) \ \big| \  Z_n>0)\big] \to \ell \ . \]
Now
\begin{align*}
\mathbf E\big[Y_n \mathbf P\big([\langle T\rangle_n]_{n-r_n}& \in B_n \mid \Pi ,Z_n > 0\big)\ ;  Z_n>0\big] \\
&=  \mathbf E\big[ Y_n \tfrac {\mathbf P([\langle T\rangle_n]_{n-r_n} \in B_n  ,Z_n > 0 \mid \Pi)}{\mathbf P(Z_n>0\mid \Pi)} \ ; Z_n >0\big]  \\
&=  \mathbf E\big[ Y_n \mathbf P([\langle T\rangle_n]_{n-r_n} \in B_n  ,Z_n > 0 \mid \Pi )\big] \\
&=  \mathbf E\big[  \mathbf E[Y_n;[\langle T\rangle_n]_{n-r_n} \in B_n  ,Z_n > 0 \mid \Pi ]\big] \\
&=   \mathbf E\big[Y_n;[\langle T\rangle_n]_{n-r_n} \in B_n  ,Z_n > 0  \big] \ .
\end{align*}
This  gives the claim of Theorem \ref{trest}.
\end{proof}

\section{Proof of Theorem 1.4}

Let again $\tilde {\mathsf T}$ denote the LPP-trest.
Recall that $\tilde Z_j^i$ is for $i<j$ the number of the individuals in generation $j$ other than $\tilde K_j$, which descent from $\tilde K_i$ but not from $\tilde K_{i+1}$.  For convenience we put $\tilde Z_j^i=0$ for $i \ge j$.

\begin{lemma} \label{le51}
Let $0 < t < 1$. Then for every $\varepsilon > 0$ there is a natural number $ a  $ such that for  any natural numbers $m$ and $ \varsigma \in [\tau_{nt}, nt]$
\[ \mathbf P\Big( \sum_{i: |i-\tau_{nt}| \ge a} \frac{\tilde Z_{\varsigma}^i}{e^{S_{\varsigma}-S_{\tau_{nt}}}}\ge \varepsilon \ \Big|\ \tau_{n-m}=n-m\Big) \le \varepsilon \ , \]
if $n$  is sufficiently large (depending on $\varepsilon, a$ and $m$). $\varsigma$ may be random, depending only on the random environment $\Pi$.
\end{lemma}

\begin{proof}  For $0<\varepsilon \le 1$ from Markov inequality  and \eqref{Ztilde}
\begin{align*}
\varepsilon \mathbf P\Big(&\sum_{ |i-\tau_{nt}| \ge a}  \frac{\tilde Z_{\varsigma}^i}{e^{S_{\varsigma}-S_{\tau_{nt}}}} \ge \varepsilon ;\tau_{n-m}={n-m}\Big) \\ &\le   \mathbf E\Big[1  \wedge \sum_{i\le \varsigma, |i-\tau_{nt}| \ge a} \eta_{i+1} e^{S_{\tau_{nt}}-S_i}  ; \tau_{n-m}=n-m\Big] \ .
\end{align*}
Next we decompose with the value of $\tau_{nt}$ to obtain for $m \le (1-t)n$
\begin{align*}
\varepsilon& \mathbf P\Big( \sum_{  |i-\tau_{nt}| \ge a}  \frac{\tilde Z_{\varsigma}^i}{e^{S_{\varsigma}-S_{\tau_{nt}}}} \ge \varepsilon ;\tau_{n-m}={n-m}\Big) \\ &\le \sum_{  j \le nt} \mathbf E\Big[1 \wedge \sum_{i \le \varsigma, |i-j| \ge a} \eta_{i+1} e^{S_j-S_i}; \tau_j=j, L_{j,nt} \ge 0\Big] \\ &\qquad\qquad\qquad \times  \mathbf P\big(\tau_{(1-t)n-m}=\lfloor (1-t)n\rfloor-m\big) \ .
\end{align*}
We split the expectation:
\begin{align*}
\sum_{  j \le nt} &\mathbf E\Big[1 \wedge \sum_{ i\le \varsigma,  |i-j| \ge a } \eta_{i+1} e^{S_j-S_i}; \tau_j=j, L_{j,nt} \ge 0\Big] = \\
& =\sum_{   j \le nt}  \mathbf E\Big[1 \wedge \sum_{i=0}^{j-a} \eta_{i+1}e^{S_j-S_i}; \tau_j=j\Big]\mathbf P( L_{nt-j} \ge 0) \\
&\qquad\qquad \mbox{} + \sum_{  j \le nt}  \mathbf P(\tau_j=j) \mathbf E\Big[1 \wedge \sum_{i=j+a}^\varsigma \eta_{i+1} e^{S_j-S_i};  L_{j,nt} \ge 0\Big] \ .
\end{align*}
Duality yields
\begin{align*}
\sum_{  j \le nt} &\mathbf E\Big[1 \wedge \sum_{i\le \varsigma,  |i-j| \ge a } \eta_{i+1} e^{S_j-S_i}; \tau_j=j, L_{j,nt} \ge 0\Big] = \\
& \le \sum_{ a \le j \le nt}  \mathbf E\Big[1 \wedge \sum_{i=a}^{j } \eta_{i }e^{S_i}; M_j<0\Big]\mathbf P( L_{nt-j} \ge 0) \\
&\qquad\qquad \mbox{} + \sum_{ a \le k \le nt}  \mathbf P(\tau_{nt-k}=\lfloor nt\rfloor -k) \mathbf E\Big[1 \wedge \sum_{i= a}^{k} \eta_{i+1} e^{ -S_i};  L_{ k} \ge 0\Big]\ .
\end{align*}
From Lemmas \ref{limitEminus}, \ref{le2} we may choose $a$ so large that
\begin{align*} \mathbf E\Big[1 \wedge \sum_{i=a}^{j } \eta_{i } e^{S_i}; M_j<0\Big] &\le \delta \mathbf P(M_j < 0) \\
\mathbf E\Big[1 \wedge \sum_{i= a}^{k} \eta_{i+1} e^{ -S_i};  L_{ k} \ge 0\Big] &\le \delta \mathbf P(L_k \ge 0)
\end{align*}
for all $j,k>a$ and given $\delta >0$.
It follows from duality
\begin{align*} \sum_{  j \le nt}  &\mathbf E\Big[1 \wedge \sum_{ i \le \varsigma, |i-j| \ge a } \eta_{i+1} e^{S_j-S_i}; \tau_j=j, L_{j,nt} \ge 0\Big]  \\& \le \delta \sum_{ a \le j \le nt}  \mathbf P(\tau_j=j)\mathbf P( L_{nt-j} \ge 0) \\
&\qquad\qquad \mbox{} + \delta \sum_{ a \le k \le nt}  \mathbf P(\tau_{nt-k}=\lfloor nt\rfloor-k) \mathbf P( L_{ k} \ge 0) \le 2 \delta 
\end{align*}
and 
\begin{align*}
\mathbf P\Big(  \sum_{  |i-\tau_{nt}| \ge a} \frac{\tilde Z_{\varsigma}^i}{e^{S_{\varsigma}-S_{\tau_{nt}}}}&\ge \varepsilon;\tau_n=n\Big)\\
&\le \frac{2\delta}\varepsilon  \mathbf P\big(\tau_{(1-t)n-m}= \lfloor(1-t)n\rfloor-m\big) \ . 
\end{align*}
Since $\mathbf P(\tau_n=n)$ is regularly varying, the right-hand side is bounded by the term $\varepsilon \mathbf P(\tau_n=n)$, if $\delta$ is chosen small enough. This gives the claim.
\end{proof}

We now come to the proof of the first part of Theorem \ref{theomain}. Let
$ \sigma_{i,n}$ as in \eqref{sigma} 
and define $\mu_n(i)$ as the smallest natural number $j$ between $1$ and $i$ such that $ \tau_{nt_i}= \sigma_{j,n} $,  
\begin{align}\mu_n(i)= \min\{ j \le i : \tau_{nt_i}= \sigma_{j,n}  \} \ . 
\label{mun}
\end{align}
Again let $\tilde Z_j$ be the number of individuals in generation $j$ of the LPP-trest $\tilde{\mathsf T}$, thus
\[ \tilde Z_j = 1+ \sum_{k=0}^{j-1} \tilde Z_j^k \ . \]
Therefore, given $ \varepsilon >0 $ in view of the preceding lemma with $\varsigma= \tau_{nt}$ there is a natural number $a$ such that given $\tau_{n-m} = n-m$ the probability  is at least $1-\varepsilon$ that the event   
\[ \tilde Z_{\tau_{nt_i}} = 1 + \sum_{|k-\tau_{nt_i}| \le a}\tilde  Z_{\tau_{nt_i}}^k=1 + \sum_{k=\sigma_{\mu_n(i),n}-a}^{\sigma_{\mu_n(i),n}} \tilde Z_{\sigma_{\mu_n(i),n}}^k  \]
holds for all $i=1, \ldots, r$. Now note that given the environment $\Pi$ the distribution of
\[ 1 + \sum_{k=\sigma_{j,n}-a} ^{\sigma_{j,n}}\tilde Z_{\sigma_{j,n}}^k \]
only depends on $(Q_{\sigma_{j,n}-a}, \ldots, Q_{\sigma_{j,n}})$. Lemma \ref{cor} says that given $\tau_{n-m}=n-m $ these random vectors are asymptotically i.i.d. Also this lemma gives asymptotic independence of these random variables from 
\[ \tfrac 1{a_n} (S_{\sigma_{1,n}}, S_{nt_1},\ldots,S_{\sigma_{r,n}}, S_{nt_r})\ , \]
which in turn determines $\mu_n(1),\ldots, \mu_n(r)$. Finally in view of Lemma \ref{funclimit} $\mu_n(1),\ldots, \mu_n(r)$ converges in distribution to $\mu(1), \ldots, \mu(r)$. 

These observations hold for every $\varepsilon >0$. Therefore we may summarize our discussion as follows: For all $m \ge 1$
\begin{align*}
\big(( \tilde Z_{\tau_{nt_1}},\ldots,\tilde Z_{\tau_{nt_r}}) \mid \tau_{n-m}=n-m \big) \stackrel{d}{\to} (V_{\mu(1)}, \ldots, V_{\mu(r)}) \ ,
\end{align*}
where the right-hand term has just the properties as given in Theorem \ref{theomain}. Now Theorem \ref{trest} gives the claim.\\\\
The proof of the second part of Theorem \ref{theomain} is prepared by the following lemma. Let for fixed $a$
\[  \hat Z_{a,k} = \sum_{i:|i- \tau_{nt}| \le a} \tilde Z_{k}^i  \]
and
\[ \alpha_{a,n}=e^{S_{\tau_{nt}}-S_{nt}} \hat Z_{a,nt}  \ , \  \beta_{a,n}= e^{S_{\tau_{nt}}-S_{\tau_{nt}+a}} \hat Z_{a, \tau_{nt}+a}\ . \]

\begin{lemma} \label{le52}
Let $m \ge 1$, $\varepsilon >0$ and $0<t<1$. Then, if $a$ is sufficiently large
\[ \limsup_{n \to \infty} \mathbf P( |\alpha_{a,n}-\beta_{a,n}| > \varepsilon \mid \tau_{n-m}=n-m) \le \varepsilon\ . \]
\end{lemma}

\begin{proof} Because of  Markov inequality and \eqref{Ztilde}
\begin{align*}
\mathbf P( \beta_{a,n}&> d \mid \tau_{n-m}=n-m)\\
&\le \mathbf P (e^{S_{\tau_{nt}}-S_{\tau_{nt}+a}}  \mathbf E[\hat Z_{a,\tau_{nt}+a} \mid \Pi] > \sqrt d \mid \tau_{n-m}=n-m) + \frac 1{\sqrt d}\\
&\le \mathbf P \Big( \sum_{i:|i- \tau_{nt}| \le a}  \eta_{i+1} e^{S_{\tau_{nt}}-S_i} > \sqrt d \ \Big|\ \tau_{n-m}=n-m\Big) + \frac 1{\sqrt d}\ .
\end{align*}
From Lemma  \ref{cor} (with $r=1$, thus $\sigma_{1,n}=\tau_{nt}$) it follows that the sum converges in distribution for $n \to \infty$ and 
\begin{align*}
\limsup_{n \to \infty} \mathbf P( \beta_{a,n}&> d \mid \tau_{n-m}=n-m) \\
&\le \mathbf P^-\Big( \sum_{i \ge 1} \eta_i e^{S_i} \ge \sqrt d\Big) + \mathbf P^+\Big( \sum_{i \ge 0} \eta_{i+1} e^{-S_i} \ge \sqrt d\Big) + \frac 1{\sqrt d} \ . 
\end{align*}
Therefore from Lemma \ref{le2} it results that there is a $d<\infty$ such that for all $a>0$
\[ \limsup_{n \to \infty} \mathbf P(  \beta_{a,n} > d \mid \tau_{n-m}=n-m) < \varepsilon/2 \ . \]
Moreover  from Lemma \ref{funclimit}  $t- \tfrac 1n \tau_{nt}$ converges in distribution to a strictly positive random variable, thus $\mathbf P(\tau_{nt}+a \ge nt \mid \tau_{n-m}=n-m) \to 0$ for $n \to \infty$. Therefore
\begin{align*}
\mathbf P(& |\beta_{a,n}-\alpha_{a,n}| > \varepsilon \mid \tau_{n-m}=n-m) \\ &\le \frac \varepsilon 2 + \mathbf P\big( |\alpha_{a,n}-\beta_{a,n}| > \varepsilon, \beta_{a,n} \le d,\tau_{nt}+a \le nt \mid \tau_{n-m}=n-m\big) \ .
\end{align*}

Now, given $\Pi$, $\hat Z_{a,\tau_{nt}+a}$ and $\tau_{nt}+a \le nt$, the process $\hat Z_{a,k}$, $k \ge \tau_{nt}+a$ is a branching process in varying environment. 
Therefore $\mathbf E[ \alpha_{a,n} \mid \Pi, \hat Z_{a,\tau_{nt}+a}] =  \beta_{a,n}  $ a.s. Also the branching property  yields
\begin{align}\label{variance}
\frac{\mathbf {Var}(Z_n \mid Z_0=z,\Pi)}{\mathbf E[Z_n \mid Z_0=1,\Pi]^2} =z\Big( e^{-S_n}+\sum_{i=0}^{n-1} \eta_{i+1}e^{-S_i}-1\Big) \ , 
\end{align}
therefore on $\tau_{nt}+a \le nt$
\begin{align*}
\varepsilon^{ 2}\mathbf P( &|\beta_{a,n}-\alpha_{a,n}| > \varepsilon \mid \Pi, \hat Z_{a,\tau_{nt}+a})  \le \mathbf E[(\beta_{a,n}-\alpha_{a,n})^2 \mid \Pi, \hat Z_{a,\tau_{nt}+a}]\\
& \le \hat Z_{a,\tau_{nt}+a}  e^{2(S_{\tau_{nt} }-S_{\tau_{nt}+a})}\Big(e^{-(S_{nt}-S_{\tau_{nt}+a})}+\sum_{i=\tau_{nt}+a}^{\lfloor nt\rfloor} \eta_{i+1}e^{-(S_i-S_{\tau_{nt}+a} )}\Big)\\
&= \beta_{a,n} \Big(e^{-(S_{nt}-S_{\tau_{nt}})}+ \sum_{i=\tau_{nt}+a}^{\lfloor nt\rfloor} \eta_{i+1}e^{-(S_i-S_{\tau_{nt} } )}\Big)
\end{align*}
Inserting this estimate we obtain
\begin{align*}
\mathbf P(& |\beta_{a,n}-\alpha_{a,n}| > \varepsilon ; \tau_{n-m}=n-m) \\
&\le \frac \varepsilon 2 \mathbf P(\tau_{n-m}=n-m)+ \frac{d}{\varepsilon^2} \mathbf E\Big[1\wedge \Big(e^{-(S_{nt}-S_{\tau_{nt}})}\\
&\qquad\qquad \mbox{} + \sum_{i=\tau_{nt}+a}^{\lfloor nt\rfloor} \eta_{i+1}e^{-(S_i-S_{\tau_{nt} } )}\Big); \tau_{nt}+a \le nt  , \tau_{n-m}=n-m\Big]\\
&\le \frac \varepsilon 2 \mathbf P(\tau_{n-m}=n-m)+ \frac{d}{\varepsilon^2} \sum_{j\le nt-a}\mathbf P( \tau_{nt}=j) \mathbf E\Big[1\wedge \Big(e^{-S_{nt-j}}\\
&\qquad\qquad \mbox{} + \sum_{i= a}^{\lfloor nt\rfloor-j} \eta_{i+1}e^{-S_i}\Big)  ; L_{nt-j}\ge 0\Big] \mathbf P\big(\tau_{n(1-t)-m}=\lfloor n(1-t)\rfloor-m\big)\ .
\end{align*}

From Lemmas \ref{pro1}, \ref{limitEminus}, \ref{le2} together with the fact that $\mathbf P(\tau_n=n)$ is regularly varying our claim follows for $a$ sufficiently large.
\end{proof}

\noindent
We are now ready to finish the proof of Theorem \ref{theomain}. We first treat the case $r=1$. From $\tilde Z_{nt}=1 + \hat Z_{a,nt}+ \sum_{i: |i-\tau_{nt}|> a}  \tilde Z_{nt}^i  $
\begin{align*}
\mathbf P\big(  | e^{S_{\tau_{nt}}-S_{nt}} &\tilde Z_{nt}- \beta_{a,n} | \ge 3\varepsilon \mid \tau_{n-m}=n-m)\\
&\le \mathbf P( e^{S_{\tau_{nt}}-S_{nt}} \ge \varepsilon\mid \tau_{n-m}=n-m ) \\
&\mbox{}\quad + \mathbf P( |\alpha_{a,n}-\beta_{a,n}| \ge \varepsilon \mid \tau_{n-m}=n-m) \\
&\mbox{}\quad  + 
\mathbf P\Big( e^{S_{\tau_{nt}} -S_{nt}} \sum_{i: |i-\tau_{nt}|> a}  \tilde Z_{nt}^i \ge \varepsilon\ \Big|\ \tau_{n-m}=n-m\Big) \ .
\end{align*}

From Lemma \ref{funclimit} it results that 
\[ \mathbf P( e^{S_{\tau_{nt}}-S_{nt}} \ge \varepsilon \mid \tau_{n-m}=n-m)= \mathbf P\big( \tfrac {S_{\tau_{nt}}-S_{nt}}{a_n} \ge \tfrac{\log \varepsilon}{a_n} \mid \tau_{n-m}=n-m\big) \to 0 \ . \] 
Together with Lemmas \ref{le51}, \ref{le52}  it follows that for all $\varepsilon >0$ there is a natural number $a$ such that 
\[ \mathbf P\big( | e^{S_{\tau_{nt}}-S_{nt}} \tilde Z_{nt}- \beta_{a,n}| \ge 3\varepsilon \mid \tau_{n-m}=n-m\big) \le 3\varepsilon \]
for large $n$.

Now from Lemma \ref{cor} we see that $\beta_{a,n}$,   conditioned on $\tau_{n-m}=n-m$,  converges in distribution for every $a$. This implies that $e^{S_{\tau_{nt}}-S_{nt}} \tilde Z_{nt}$ conditioned on $\tau_{n-m}=n-m$ converges in distribution.
Moreover from  Lemma \ref{le41}  there is a $\delta >0$ such that
\[ \mathbf P^+ \Big(e^{ - S_{ a}} \sum_{1  \le i \le a } \tilde Z_{  a}^i< \delta \Big) < \varepsilon \ ,\]
if only $a$ is sufficiently large. Then from Lemma \ref{cor}  
\[ \mathbf P( \beta_{a,n}< \delta \mid \tau_{n-m}=n-m ) < \varepsilon \ , \]
if only $n$ is sufficiently large. Therefore the limiting distribution of $e^{S_{\tau_{nt}}-S_{nt}} \tilde Z_{nt}$ conditioned on $\tau_{n-m}=n-m$ has no atom in zero. An application of Theorem \ref{trest} now gives the claim for $r=1$.

Finally for $r>1$ we let 
\[ \beta_{a,n,i} = e^{S_{\sigma_{i,n}}-S_{\sigma_{i,n}+a}}\hat Z_{a,\sigma_{i,n}+a} \ , \quad i=1, \ldots, r \ .\] 
From \eqref{mun} and our considerations above we know that for every $i \le r$
\[ \mathbf P\big( | e^{S_{\tau_{nt_i}}-S_{nt_i}} \tilde Z_{nt_i}- \beta_{a,n,\mu_n(i)}| \ge \varepsilon \text{ for some } i \le r \mid \tau_{n-m}=n-m\big) \le \tfrac\varepsilon r  \]
and the rest of the theorem follows by means of Lemma \ref{cor} and Theorem \ref{trest}.

%

%

%
 
\end{document}